\documentclass[reqno, 11pt]{amsart}
\pdfoutput=1
\usepackage{graphicx} 
\usepackage{stmaryrd}

\usepackage[french,english]{babel}
\usepackage{graphicx}
\usepackage{subcaption}
\usepackage{tabularx}
\usepackage{booktabs}
\usepackage{array}
\usepackage{amsmath}
\usepackage{amsfonts}
\usepackage{amssymb}
\usepackage{amsthm}
\usepackage{bbm}\usepackage{mathtools}
\usepackage{float}
\usepackage[scr]{rsfso}
\usepackage{enumitem}
\usepackage{permute}
\usepackage{slashed}
\usepackage[usenames,dvipsnames]{xcolor}
\usepackage[pagebackref = false, colorlinks, linkcolor = Red, citecolor = Green, bookmarksdepth=2, linktocpage=true]{hyperref}
\usepackage[capitalize]{cleveref}
\usepackage{verbatim}

\crefformat{footnote}{#2\footnotemark[#1]#3}

\usepackage[T1]{fontenc}

\usepackage{tikz}
\usetikzlibrary{calc}
\usetikzlibrary {arrows.meta}

\newcommand{\defeq}{\coloneqq}
\newcommand{\op}{\operatorname}

\newcommand{\Sp}{\op{Sp}}

\newcommand{\be}{\begin{equation}}
	\newcommand{\ee}{\end{equation}}
\newcommand{\Ga}{\Gamma}

\newcommand{\z}{\mathbb Z}
\newcommand{\R}{\mathbb R}

\newcommand{\Z}{\mathbb Z}

\newcommand{\ga}{\gamma}

\newcommand{\la}{\lambda}
\newcommand{\La}{\Lambda}
\newcommand{\inte}{\op{int}}
\newcommand{\ba}{\backslash}

\newcommand{\cal}{\mathcal}
\newcommand{\br}{\mathbb R}
\newcommand{\SO}{\op{SO}}
\newcommand{\irr}{\op{irr}}

\newcommand{\PSL}{\op{PSL}}

\newcommand{\bH}{\mathbb H}

\newcommand{\q}{\mathbb Q}
\newcommand{\G}{\Gamma}

\renewcommand{\frak}{\mathfrak}

\newcommand{\e}{\varepsilon}

\renewcommand{\L}{\mathcal L}
\newcommand{\fa}{\mathfrak a}

\newcommand{\lox}{\op{lox}}

\DeclareMathOperator{\SL}{SL}

\renewcommand{\u}{\mathsf{u}}

\renewcommand{\epsilon}{\e}

\newtheorem{theorem}{Theorem}[section]
\newtheorem{thm}[theorem]{Theorem}
\newcommand{\diag}{\op{diag}}
\newtheorem*{claim*}{Claim}

\newtheorem{lemma}[theorem]{Lemma}
\newtheorem{lem}[theorem]{Lemma}

\newtheorem{cor}[theorem]{Corollary}

\newtheorem{prop}[theorem]{Proposition}

\newtheorem{Con}[theorem]{Conjecture}

\theoremstyle{definition}
\newtheorem{definition}[theorem]{Definition}
\newtheorem{Def}[theorem]{Definition}

\theoremstyle{remark}

\newtheorem{Rmk}[theorem]{Remark}

\numberwithin{equation}{section}

\title[Totally real units and 
eigenvalue patterns in thin tubes]{Counting totally real units and 
eigenvalue patterns in $\SL_n(\mathbb Z)$ and $\Sp_{2n}(\mathbb Z)$ in thin tubes}

\author{Hee Oh}
\address{Department of Mathematics, Yale University, New Haven, CT 06511}
\email{hee.oh@yale.edu}

\begin{document}
\begin{abstract}
For a vector $v=(v_1,\dots ,v_n)$ with $v_1>\cdots>v_n$ and
$\sum v_i=0$, we study the {\it directional entropy} of two arithmetic objects:
\begin{enumerate}
\item the logarithmic embeddings of degree-$n$ totally real units, and
\item the logarithmic eigenvalue data of $\operatorname{SL}_n(\mathbb Z)$.
\end{enumerate}
In each case, the entropy in the direction of $v$ is
$$
 \mathsf E_n(v)= \rho_{\operatorname{SL}_n}(v)=\sum_{i=1}^{n-1}(n-i)\,v_i,
$$
the value of the half-sum  of positive roots of $\operatorname{SL}_n(\mathbb R)$ evaluated at $v$.
More precisely, the number of objects lying in a thin tube around the ray $\br_+v$ and of norm at most $T$ grows on the order of $ \exp\!\bigl(\rho_{\operatorname{SL}_n}(v)\,T\bigr)$ as $T\to \infty$.

Because each eigenvalue data determines an
$\operatorname{SL}_n(\mathbb R)$-conjugacy class, this implies a lower bound of order
$\exp\!\bigl(\rho_{\operatorname{SL}_n}(v)T\bigr)$ for the number of
$\operatorname{SL}_n(\mathbb Z)$-conjugacy classes with a prescribed eigenvalue data; we also obtain  an upper bound of order
$\exp\!\bigl(2\rho_{\operatorname{SL}_n}(v)T\bigr)$.

A parallel argument for the symplectic lattice $\operatorname{Sp}_{2n}(\mathbb Z)$, taken in the symmetric direction
$v=(v_1,\dots ,v_n,-v_n,\dots ,-v_1),\quad v_1>\cdots>v_n>0,$
shows that
$$\mathsf E_{2n}^{\operatorname{Sp}}(v)=\rho_{\operatorname{Sp}_{2n}}(v)=\sum_{i=1}^n(n+1-i)v_i,$$ the half-sum of positive
roots of $\operatorname{Sp}_{2n}(\mathbb R)$.
\end{abstract}

	\maketitle
	\section{Introduction}\label{sec:Introduction}
    Classical arithmetic asks how many objects of a given kind—ideals,
points, matrices, geodesics—fit inside a region that grows without
bound.  In higher rank, the natural “size’’ of an object is rarely a
single number; instead it is a vector that records growth rates in
several directions at once.  When we restrict our attention to a thin
tube around a fixed ray, the leading exponent of an exponential growth can be viewed as a \emph{directional
entropy}: it measures how densely the arithmetic set populates that ray.

This paper pinpoints an explicit linear functional that governs the
directional entropy of the following two collections:
\begin{itemize}
  \item the logarithmic embeddings of \emph{all} totally real units of
        fixed degree~$n$;
  \item the logarithmic eigenvalue data (=Jordan projections) of 
        elements in\/ $\SL_n(\mathbb Z)$.
\end{itemize}

We prove that both collections exhibit the same entropy along every ray in the positive Weyl chamber.  
 Going further, we count
$\SL_n(\mathbb Z)$-conjugacy classes that share a prescribed eigenvalue
pattern. The Jordan-projection entropy
yields an immediate lower bound for this count, and we provide an upper bound, which is conjecturally a true order of magnitude. We also address the analogous problem for the symplectic lattice $\Sp_{2n}(\z)$.

\subsection*{Totally real algebraic units} For an integer $n\ge 2$, let $\cal K^\star_n$ denote
the set  of totally real number
fields $K$ of degree $n$. 
For each $K\in \cal K^\star_n$, let $\Sigma_K$ denote the set of all {\it ordered} embeddings of $K$ into $\br$.
Define
$$\cal K_n=\bigsqcup_{K\in \cal K_n^\star, \sigma\in \Sigma_K} (K, \sigma) .$$  For $(K,\sigma) \in \cal K_n$ with $\sigma=(\sigma_1, \cdots, \sigma_n)$, define the logarithmic map
$$\Lambda_{K,\sigma} : K-\{0\}\to \br^n,\;\; 
\Lambda_{K,\sigma} (\mathsf u)=
   \bigl(
   \log\lvert\sigma_{1}(\u)\rvert,\dots,
   \log\lvert\sigma_{n}(\u)\rvert
   \bigr).$$
 Denote by $O_{K}$ the ring of integers of $K$ and by
$O_{K}^\times$ its unit group.
Consider the hyperplane $$\mathsf H=\{ v=(v_1, \cdots, v_n)\in \br^n: \sum v_i=0\}$$
and note that, by Dirichlet's unit theorem, $\Lambda_{K,\sigma} (O_K^\times)$ is a lattice in $\mathsf H$ (cf. \cite{Ko}).

We extend $\Lambda_{K,\sigma}$  coordinate-wise to a map
$\Lambda:\cal K_n-\{0\}\to \br^n $: $\La(u)=\La_{K, \sigma}(u)$ if $u\in (K, \sigma)$.  
Collect all totally
real units of degree $n$ in the disjoint union
$$
O_n^\times =\bigsqcup_{(K,\sigma)\in \cal K_n} (O_{K}^\times,  \sigma ).
$$

\medskip 
We propose the following notion of directional entropy for  vectors in $$
\mathsf H_+ =\Bigl\{v=(v_1,\dots,v_n)\in\mathsf H:
            v_1>\dots>v_n \Bigr\}.
$$

\begin{definition}[Directional entropy] Fix a norm $\|\cdot\|$ on $\br^n$. For $v\in \mathsf H_+$,  
define the \emph{upper} and \emph{lower} directional entropies of $O_n^\times $ in the direction $v$ by
\begin{align*}
     \overline{\mathsf E}_{n}(v)
   &  :=\|v\|\cdot  \lim_{\varepsilon\to0}\;
       \limsup_{T\to\infty}\frac1T\log N_{\varepsilon}(T,v),
  \\ 
  \underline{\mathsf E}_{n}(v)
   &  :=\|v\|\cdot \lim_{\varepsilon\to0}\;
       \liminf_{T\to\infty}\frac1T\log N_{\varepsilon}(T,v)
\end{align*}
where $$ N_{\varepsilon}(T,v):=
  \#\bigl\{u\in O_n^{\times}:\; \|\Lambda(u)\|\le T,\;
          \|\Lambda(u)-\br_+v\|<\varepsilon\bigr\}.$$ 
          
These quantities lie in $\{-\infty\} \cup [0,\infty)$,  are independent of the choice of a norm, and are homogeneous of degree one in $v$.
When they coincide, we write $\mathsf E_{n}(v)$ for their common value.
\end{definition}

\begin{Rmk} We call this quantity a {\it directional entropy} because it records the exponential growth rate of units whose logarithmic embeddings stay in a thin tube about the ray $\mathbb R_{+}v$, mirroring standard entropy-type counts in dynamics.
\end{Rmk}
 
We compute the directional entropy of $O_n^\times$ for every $v\in \mathsf H_+$: \begin{thm}\label{iup3}
    For each $v=(v_1, \cdots, v_n)\in \mathsf H_+$, we have
    $$
    \mathsf E_n(v) = \sum_{i=1}^{n-1} (n-i)v_i.
    $$
\end{thm}
We expect the same formula to hold
for vectors lying on the  walls of $\mathsf H_+$.

The entropy $\mathsf E_n(v)$ may also be expressed via the discriminant of the model polynomial
 $q_{Tv}(x)=\prod_{i=1}^n (x-e^{Tv_i})$:
 $$\lim_{T\to\infty}\frac{1}{T} \log \sqrt{\op{Disc} \, (q_{Tv})} =\sum_{i=1}^{n-1} (n-i)v_i. 
    $$

\begin{Rmk} \label{max}
On  the unit sphere for the max-norm
$\{v\in \mathsf H_+: \|v\|_{\max}=1\}$, the entropy functional $\mathsf E_n$ reaches its supremum
$\lfloor \frac{n^2}4\rfloor$ in the direction of
$v=(1,\cdots 1, 0, -1,  \cdots, -1)$ for  $n$ odd 
 and $v=(1,\cdots, 1, -1, \cdots, -1)$ for $n$ even, where
 the first $\lfloor n/2\rfloor$-coordinates are $1$. For instance,
 $$\sup\{\mathsf E_4 (v):v\in \mathsf H_+, \|v\|_{\max}=1\} =4.$$
 On the Euclidean unit sphere
 $\{\|v\|_{\op{Euc}}=1\}$,
the maximum value of $\mathsf E_n$ is  $ \sqrt{\frac{n(n^2-1)}{12}}$, attained
in the direction $(n-1, n-3, \cdots, -(n-3),-(n-1) )$.
\end{Rmk}

\begin{figure} 
  \includegraphics [height=4.5cm]{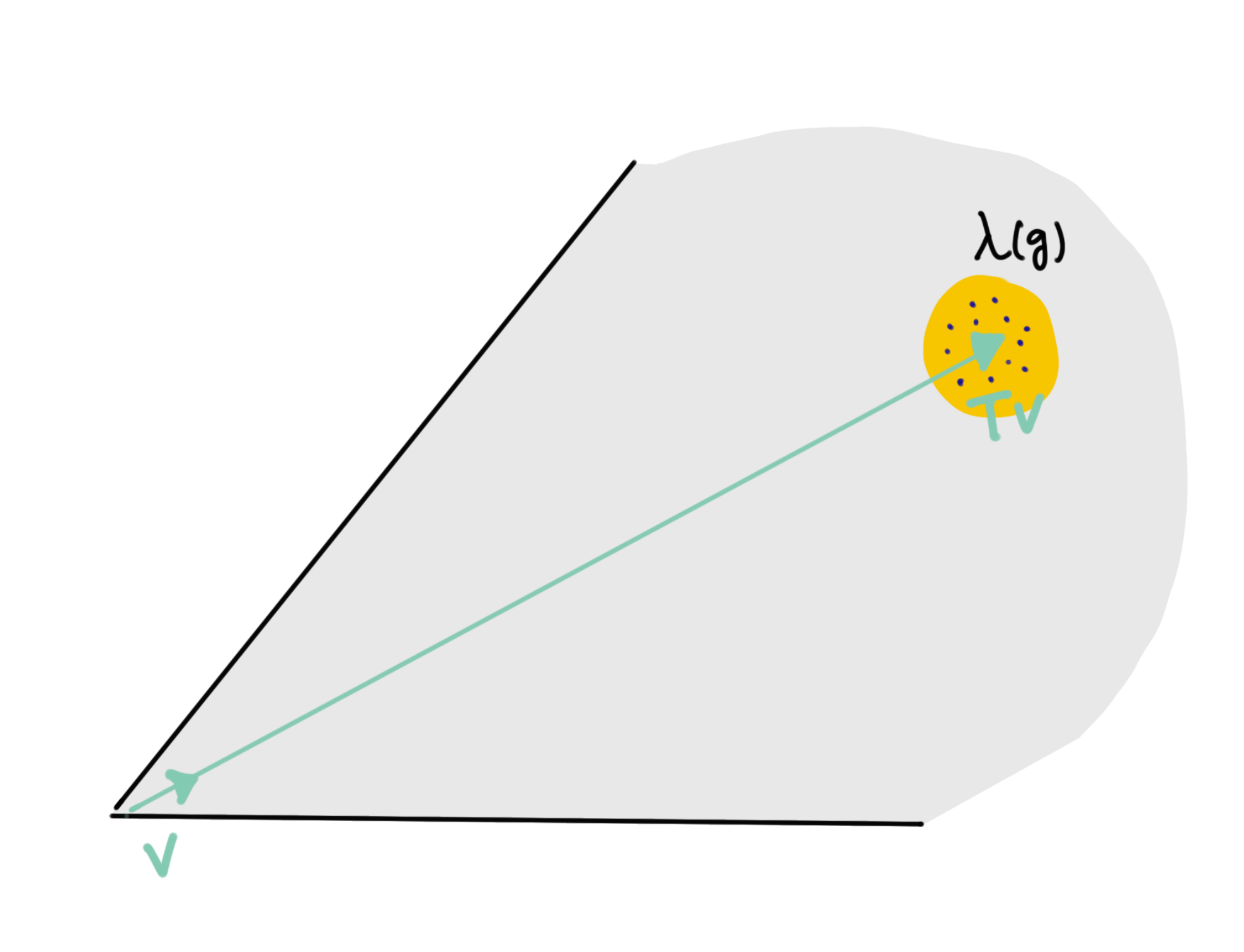}  
  \includegraphics [height=4.5cm]{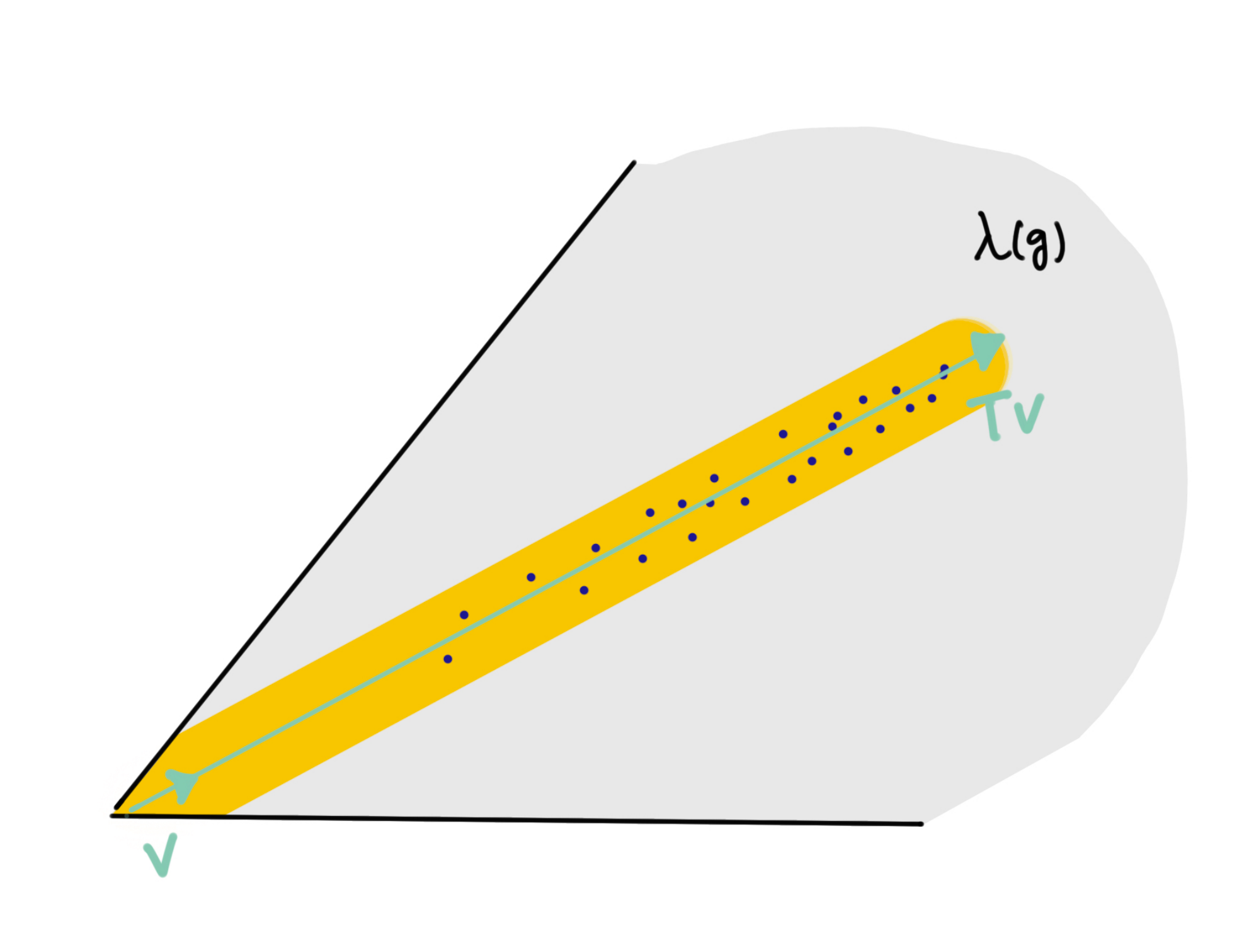}  
\end{figure}

The following quantitative theorem yields Theorem \ref{iup3}.
\begin{thm} \label{ms} Let $v\in \mathsf H_+$. For all sufficiently small $\e>0$, we have 
$$ \# \bigl\{ \u\in   O_n^\times :
   \| \Lambda( \u)- T v \|_{\max} <\e  \bigr\}\;\; \asymp_\e\footnote{We write $f(T)\asymp g(T)$ if there exist $C_1,C_2>0$  such that
$C_1\,g(T)\le f(T)\le C_2\,g(T)$ for all $T\ge 1$. The notation $f(T)\asymp_{\varepsilon}g(T)$ has the same meaning, except that
$C_1$ and $C_2$ may depend on $\varepsilon$.} \;\; \exp \left( \sum_{i=1}^{n-1} (n-i)v_iT \right) .$$
More precisely,
   \begin{multline*}   2 \left(\frac{4\e}{ (n-1)3^n} \right)^{n-1} 
  \le \liminf_{T\to \infty}  \frac{\# \bigl\{ \u\in   O_n^\times :
   \| \Lambda( \u)- T v \|_{\max} <\e  \bigr\}}{\exp {\left( \sum_{i=1}^{n-1}(n-i)v_i T \right) }}  \\ \le \limsup_{T\to \infty} \frac{\#  \bigl\{ \u\in   O_n^\times :
   \| \Lambda( \u)- T v \|_{\max}<\e  \bigr\}}{\exp { \left(\sum_{i=1}^{n-1}(n-i)v_i T\right) }}  \le 2 (4\e)^{n-1} n! .\end{multline*} 
\end{thm}  

It is natural to ask whether the limit 
$\lim_{T\to \infty} \frac{\# \bigl\{ \u\in   O_n^\times :
   \| \Lambda( \u)- T v \|_{\max} <\e  \bigr\}}{\exp {\left( \sum_{i=1}^{n-1}(n-i)v_i T \right) }}$ exists, and, if so, what its value is.  In Proposition \ref{um}, we give exponential error terms in the upper and lower bounds above; in particular, these error terms can be taken uniformly over all $v$ in any fixed compact subset of $\mathsf H_+$.

\subsection*{Eigenvalue patterns in $\SL_n(\z)$} 
An element $g\in \SL_n(\br)$ is called {\it loxodromic} if its eigenvalues have pairwise distinct moduli; in particular, they are all real.
For such  $g\in\SL_n(\mathbb R)$, write its eigenvalues as
\be\label{eg} \cal E(g)= \left( m_1(g) e^{\lambda_1(g)},\dots,m_n(g) e^{\lambda_n(g)}\right) \ee
with signs $m_i(g)\in \{\pm 1\}$ and ordering given by
$\lambda_1(g)\ge\cdots\ge\lambda_n(g)$.
Set 
\be\label{eg2}
  \lambda(g)\coloneqq
     \bigl(\lambda_{1}(g),\dots,\lambda_{n}(g)\bigr),
  \qquad
  m(g)\coloneqq\bigl(m_{1}(g),\dots,m_{n}(g)\bigr).
\ee 
The vector $\lambda(g)$ is called the {\it Jordan projection} of $g$.
Define the linear functional
$$\rho_{\SL_n}(v)=\tfrac{1}{2}\sum_{1\le i<j\le n} (v_i-v_j)= \sum_{i=1}^{n-1} (n-i)v_i$$
 which is the half-sum of all positive roots of $\SL_n(\br)$.
For $v\in \mathsf H_+$ and $\e>0$, if $T$ is sufficiently large, then any $\ga\in \SL_n(\br)$ with $\|\lambda(\ga)-Tv\|<\e$ is loxodromic.

We prove the following counting results: \begin{thm} \label{strong}
Let $v\in\mathsf H_+$ and let $m=(m_1, \cdots, m_n)\in \{\pm 1\}$ be a sign pattern with $\prod_{i=1}^n m_i=1$.
Fix $\e>0$.
\begin{enumerate}
    \item 
We have 
$$
\#\Bigl\{ \lambda (\ga):\ga\in \SL_n(\z),
              \|\lambda(\ga)-Tv\|\le\e, m(\ga)=m\Bigr\}  \asymp_\e  e^{\rho_{\SL_n}(v)T} ,$$
              where explicit upper and lower
            multiplicative constants are given in Theorem \ref{jm}.
       \item  There exist $C_1, C_2>0$ such that for all sufficiently large $T>1$, 
$$
            C_1 e^{\rho_{\SL_n}(v)T}   \le 
\#\Bigl\{ [\ga] \in [\SL_n(\mathbb Z)] :
              \|\lambda(\ga)-Tv\|\le\e, m(\ga)=m\Bigr\} \le C_2
              e^{2 \rho_{\SL_n}(v)T} $$
 where $[\SL_n(\mathbb Z)]$ denotes the set of all $\SL_n(\z)$-conjugacy classes. 
        \end{enumerate}
        \end{thm}

        Observe that $2\rho_{\SL_n}(v)$ is precisely the volume growth exponent
for the thin tubes around the ray $\mathbb R_{+}v$: if $\mu(g)\in \mathsf H_+$ is the Cartan projection of $g$, i.e., the unique element such that
$e^{\mu(g)}\in \SO(n) g \SO(n)$,
$$
  \operatorname{Vol}\Bigl\{
     g\in\SL_n(\mathbb R):\,
     \|\mu(g)-T v\|\le\varepsilon
  \Bigr\} \asymp_\e e^{\,2\rho_{\SL_n}(v)T},
$$ 
with volume taken with respect to a Haar measure of $\SL_n(\br)$
(cf.\ the proof of Theorem~\ref{cartan}).
Consequently, the Jordan–projection count in part~(1) grows like the
\emph{square root} of this ambient volume growth.

Because each eigenvalue pattern determines an
$\SL_n(\R)$-conjugacy class, this yields a lower bound of order
$e^{\rho_{\SL_n}(v)T}$, as stated in part (2), for the number of
$\SL_n(\Z)$-conjugacy classes with prescribed eigenvalue pattern; we also obtain  an upper bound of order $e^{2\rho_{\SL_n}(v)T}$.

Similar to the directional entropy for $O_n^\times$, we also propose the following notion of the directional entropies for $\SL_n(\z)$:
\begin{definition}[Directional entropy for $\SL_n(\mathbb Z)$]\label{dir}
Let $v\in\mathsf H_+$ and let $m=(m_1, \cdots, m_n)\in \{\pm 1\}$ satisfy $\prod_{i=1}^n m_i=1$.
Define the \emph{upper} and \emph{lower} directional entropies by
\begin{align*}
  \overline{\mathsf E}_{\SL_n(\mathbb Z)}(v,m)
  &  :=\|v\|\cdot  \lim_{\varepsilon\to0}\;
      \limsup_{T\to\infty}\frac{\log N_\varepsilon(T,v, m)}{T}; \\
  \underline{\mathsf E}_{\SL_n(\mathbb Z)}(v, m)
 &   :=\|v\|\cdot  \lim_{\varepsilon\to0}\;
      \liminf_{T\to\infty}\frac{\log N_\varepsilon(T,v, m)}{T}
\end{align*}
where  $$N_\varepsilon(T,v,m)\:=
     \#\bigl\{\lambda(\gamma):\gamma \in\SL_n(\mathbb Z):
            \|\lambda(\gamma)-\br_+v\|\le\varepsilon,   \|\lambda(\gamma)\|\le T, m(\ga)=m\bigr\}.$$
Similarly, set
\begin{align*}
  \overline{\mathsf E}^{\!\star}_{\SL_n(\mathbb Z)}(v, m) &
    :=\|v\|\cdot  \lim_{\varepsilon\to0}\;
      \limsup_{T\to\infty}\frac{\log M_\varepsilon(T,v,m)}{T}; \\
  \underline{\mathsf E}^{\!\star}_{\SL_n(\mathbb Z)}(v, m)
   & :=\|v\|\cdot  \lim_{\varepsilon\to0}\;
      \liminf_{T\to\infty}\frac{\log M_\varepsilon(T,v,m)}{T}
\end{align*}
where $$M_\varepsilon(T,v,m):=
     \#\bigl\{[\gamma]\in[\SL_n(\mathbb Z)]:
              \|\lambda(\gamma)-\br_+v\|\le\varepsilon,  \|\lambda(\gamma)\|\le T, m(\ga)=m\bigr\}.$$
As before, these quantities in $\{-\infty\}\cup [0, \infty)$ are norm-independent and homogeneous of degree one.
When the lower and upper limits agree, we write
$\mathsf E_{\SL_n(\mathbb Z)}(v)$ and
$\mathsf E^{\!\star}_{\SL_n(\mathbb Z)}(v)$, respectively.
\end{definition}

As an immediate consequence of Theorem \ref{strong},
we get \begin{theorem}\label{ab}
Let $v\in\mathsf H_{+}$ and
$m\in\{\pm1\}^{n}$ with
$\prod_{i=1}^{n}m_{i}=1$.
Then
$${\mathsf E}_{\SL_n(\mathbb Z)}(v, m)
    =\rho_{\SL_n}(v) ;$$
$$
\rho_{\SL_n}(v)
    \;\le\;
    \underline{\mathsf E}^{\!\star}_{\SL_n(\mathbb Z)}(v,m)
    \;\le\;
    \overline{\mathsf E}^{\!\star}_{\SL_n(\mathbb Z)}(v,m)
    \;\le\;
    2 \rho_{\SL_n}(v).
$$
\end{theorem}

We think that $\mathsf E^\star_{\SL_n(\z)}(v, m)= 2 \rho_{\SL_n}(v)$ should be true, although we do not know how to prove this; see Conjecture \ref{cj} for a more general formulation.

\medskip
Denote by
$\SL_n(\z)_{\lox}$ (resp. $[\SL_n(\z)]_{\lox}$) the set  (resp. the  set of $\SL_n(\z)$-conjugacy classes) of  loxodromic elements of $\SL_n(\z)$. For $\ga\in \SL_n(\z)_{\lox}$,  write
$$[\ga]_\br\subset \SL_n(\z) \quad\quad\text{and}\quad\quad [\ga]_\q\subset \SL_n(\z)$$
for its  $\SL_n(\br)$- and $\SL_n(\q)$-conjugacy classes, respectively.
Because the centralizer of $\ga\in \SL_n(\z)_{\lox}$ is a maximal $\q$-split torus and all such tori are conjugate under $\SL_n(\q)$, we have
$$[\ga]_\br =[\ga]_\q .$$
Define the ``class number''
\be\label{hhh} h ( \gamma )=\# \{\SL_{n}(\mathbb
Z)\text{-conjugacy classes  inside $[\ga]_\q$}\} .\ee 
Since the eigenvalue pattern of a loxodromic element uniquely determines its $\SL_n(\br)$-conjugacy class, the map $\gamma\mapsto \cal E(\ga)$ gives a bijection
$$ \{[\ga]_\br: \ga\in \SL_n(\z)_{\lox}\} \Leftrightarrow \{\cal E(\ga): \ga\in \SL_n(\z)_{\lox}\}.$$
Hence for any region $R\subset \br^n$, 
$$
  \#\bigl\{
      [\gamma] \in [\SL_n(\z)]_{\lox}: \cal E(\gamma)\in R
  \bigr\}
  \;=\; \sum_{ \cal E(\gamma)\in R} h (\ga).
$$
In other words, the number of
$\SL_n(\z)$-conjugacy classes
whose eigenvalue pattern lies in $R$ equals the count of
those patterns, each weighted by its class number $h(\ga)$.

\begin{Rmk}    In view of Theorem \ref{strong},    one may expect that for any $\e>0$,    $$h ( \gamma )  \ll_\e e^{\rho_{\SL_n}(\lambda(\ga))(1+\epsilon)} $$      for all loxodromic $\ga\in \SL_n(\z)$.\end{Rmk}

For  any $\ga\in \SL_2(\z)$ with
$D=\op{tr}(\ga)^2-4$ square-free, the quantity
 $h(\ga)$ coincides with the classical class number $h_K=\#\op{Cl}(O_K)$
 of the quadratic field $K=\q(\sqrt D)$ (see \cite{LM}, \cite{Ta}, \cite{Ko}).
Moreover, the conjugacy classes $[\ga] \in [\SL_n(\z)]_{\lox}$  correspond bijectively to closed geodesics $C_\ga$ on the modular surface $\SL_2(\z)\ba \bH^2$, with length given by $ 2\lambda_1(\ga) $ \cite{Sa}. 
Hence the prime geodesic theorem on modular surface (\cite{Se}, \cite{He}) implies
$$\#\Bigl\{ [\ga] \in [\SL_2(\mathbb Z)]_{\op{lox}} :
             T-\e\le  \|\lambda (\ga) \|<T+\e   \Bigr\} \asymp_\e \frac{  e^{2T}}{2T} .$$
  In this case, Theorem \ref{strong} gives 
$$\#\Bigl\{ \cal E(\ga):\ga  \in \SL_2(\mathbb Z)_{\op{lox}} :
             T-\e\le  \|\lambda (\ga) \| \le T +\e \Bigr\} \asymp_\e  e^T,$$
              which also follows from the  elementary fact that
              $e^{\|\lambda(\ga) \|}$ is essentially the size of the (integral)
              trace of $\ga$.
\begin{Rmk}
Eskin-Mozes-Shah studied a \emph{transversal} counting
problem in \cite{EMS}.  Fix a loxodromic element $\gamma_{0}\in\SL_{n}(\mathbb Z)$ and
let $p\in\mathbb Z[x]$ be its characteristic polynomial.  Write $K=\mathbb Q(\alpha)$ for a root $\alpha$ of $p$. Assume that $p$ is irreducible over $\q$ and $\mathbb Z[\alpha]=\mathcal O_{K}$.
By \cite[Theorem 1.1]{EMS}, as $T\to\infty$,
$$
  \#\Bigl\{
      \gamma\in[\gamma_{0}]_{\mathbb R} :
      \|\gamma\|<e^{T}
    \Bigr\}
  \;\sim\footnote{We write $f(T)\sim g(T)$ if $\lim_{T\to \infty} f(T)/g(T)=1$.}\;
  c_{n}\,
  \frac{h(\gamma_{0})\,R_{K}}{\sqrt{\operatorname{Disc}(p)}}\;
  \exp \left({\tfrac12(n^{2}-n)T}\right),
$$
where
\begin{itemize}
  \item $c_{n}>0$ depends only on~$n$;
  \item $h(\gamma_{0})$ is the class number defined in~\eqref{hhh};
  \item
    $R_{K}$ is the regulator of~$K$, i.e.\ the volume of
    $\mathsf H/ \left(\Lambda_{K,\sigma}\bigl(O_{K}^{\times}\right))$;
\end{itemize}

Thus \cite{EMS} counts \emph{integral matrices lying inside a \emph{fixed}
$\SL_{n}(\mathbb R)$–conjugacy class}, whereas our results count
\emph{the number of \emph{distinct} $\SL_{n}(\mathbb Z)$–conjugacy
classes} whose Jordan projections fall into a given tube.

\end{Rmk}

\subsection*{Eigenvalue patterns in $\Sp_{2n}(\z)$}  We also carry out a parallel analysis for the symplectic lattice $\Sp_{2n}(\mathbb Z)$, obtaining
analogous counting results and entropy estimates.
Fix the symplectic form in~\eqref{sspp} so that
$$
  \fa^{+}
  \;=\;
  \Bigl\{
      v=\diag(v_{1},\dots,v_{n},-v_{n},\dots,-v_{1})
      :\;
      v_{1}\ge\cdots\ge v_{n}\ge 0
  \Bigr\}
$$ is
a positive Weyl chamber of $\Sp_{2n}(\mathbb R)$.
An element $g\in\Sp_{2n}(\mathbb R)$ is \emph{loxodromic} precisely when
its Jordan projection
$$
  \lambda(g)
  \;=\;
  \bigl(
     \lambda_{1}(g),\dots,\lambda_{n}(g),
     -\lambda_{n}(g),\dots,-\lambda_{1}(g)
  \bigr)
  \in\operatorname{int}\fa^{+}.
$$
For such $g$, set
$$
  m(g)=\bigl(m_{1}(g),\dots,m_{n}(g)\bigr)\in\{\pm1\}^{n},
$$
so that, for each $i$, the two real eigenvalues of $g$ are
$m_{i}(g)\,e^{\pm\lambda_{i}(g)}$.

 Let  
$$    \rho_{\Sp_{2n}}(v)=\sum_{i=1}^{n}(n+1-i)\,v_{i}$$ be the half-sum of all positive roots of $(\frak{sp}_{2n}(\br),\fa)$.
\begin{thm}\label{jm2}
  Let $v\in\operatorname{int}\fa^{+}$ and
  $m\in\{\pm1\}^{n}$.
Fix small $0<\varepsilon<1$.
\begin{enumerate}
    \item 
We have
  $$
    \#\Bigl\{
        (\lambda(\gamma),m(\gamma)):
        \gamma\in\Sp_{2n}(\mathbb Z),\;
        \|\lambda(\gamma)-Tv\|\le\varepsilon,\;
        m(\gamma)=m
      \Bigr\}
    \;\asymp_{\varepsilon}\;
    e^{\rho_{\Sp_{2n}}(v)T},
  $$   where explicit upper and lower
            multiplicative constants are given in Theorem \ref{jm2}.
\item  There exist $C_1, C_2>0$ such that for all sufficiently large $T>1$,
$$
            C_1 e^{\rho_{\Sp_{2n}}(v)T}   \le 
\#\Bigl\{ [\ga] \in [\Sp_{2n}(\mathbb Z)] :
              \|\lambda(\ga)-Tv\|\le\e, m(\ga)=m\Bigr\} \le 
              C_2 e^{2 \rho_{\Sp_{2n}}(v)T} .$$
              \end{enumerate}
              
\end{thm}

Define the directional entropies
$
  \mathsf E_{\Sp_{2n}(\mathbb Z)}(v,m)
$
and
$
  \mathsf E^{\!\star}_{\Sp_{2n}(\mathbb Z)}(v,m)
$
exactly as in Definition~\ref{dir}, with $\SL_{n}(\mathbb Z)$ replaced
everywhere by $\Sp_{2n}(\mathbb Z)$.

\begin{cor}\label{jm333}
For all $v\in\operatorname{int}\fa^{+}$ and
$m\in\{\pm1\}^{n}$,
we have 
$${\mathsf E}_{\Sp_{2n}(\mathbb Z)}(v, m)
    =\rho_{\Sp_{2n}}(v) ;$$
$$
\rho_{\Sp_{2n}}(v)
    \;\le\;
    \underline{\mathsf E}^{\!\star}_{\Sp_{2n}(\mathbb Z)}(v,m)
    \;\le\;
    \overline{\mathsf E}^{\!\star}_{\Sp_{2n}(\mathbb Z)}(v,m)
    \;\le\;
    2 \rho_{\Sp_{2n}}(v).
$$
\end{cor}

\medskip 
\noindent{\bf On the proof:} We outline the proof of Theorem \ref{strong}. The proof of Theorem \ref{ms} is entirely analogous; one simply uses
 the bijection between {\it primitive} units and their minimal polynomials.
Let $v\in \mathsf H_+$ and $m$ a sign pattern. 
We translate the geometric condition ``$\lambda(\ga)$ lies in $B_\e(Tv)$ with sign pattern $m$'' into a purely arithmetic statement about integral polynomials, and then we count those polynomials.  
For a loxodromic element $\ga\in \SL_n(\z)$, its eigenvalue pattern $\cal E(\ga)$
is equivalent to its characteristic polynomial $p_\ga(x)$.
Requiring $\lambda(\ga)\in B_\e(Tv)$ and $m(\ga)=m$ forces the roots of $p_\ga$ 
to satisfy $m_ie^{Tv_i} + O(\e e^{Tv_i})$, $1\le i\le n$.
Let $\cal Q_{T}(v,m;\e)$ denote the collection of all monic integral polynomials with this property.
Using Rouch\'e's theorem, we observe that
$p\in \cal Q_{T}(v,m;\e)$ iff each coefficient lies in an interval of length $(1+O(\e))e^{T(v_1+\cdots +v_i)}$. Hence $\cal Q_{T}(v,m;\e)$ coincides with an expanding box
$\cal P_{T}(v,m;\e)$  inside $\z^{n-1}$ whose side-lengths grow at precisely those exponential rates. 
Counting integral points in this expanding box 
  is governed by the square-root of the discriminant of the model polynomial  $q_{Tv,m}(x)=\prod_{i=1}^n (x-m_ie^{Tv_i})$ with $\op{Disc} (q_{Tv,m})\asymp e^{2\rho_{\SL_n}(v)T}$.
Exactly the same reasoning works for the symplectic lattice $\Sp_{2n}(\z)$. Here one exploits the fact that the characteristic polynomials of 
 $\Sp_{2n}(\z)$ matrices are {\it precisely} the integral monic reciprocal (palindromic) polynomials of degree $2n$  (\cite{Ya2}, \cite{MS}). Because the reciprocal property simply folds
the coefficient box in half, the counting again reduces to a volume estimate and the resulting exponent is $\rho_{\Sp_{2n}}(v)=\sum_{i=1}^n (n+1-i)v_i$.
For other arithmetic groups, no tidy description is available
for the integral polynomials that arise as
characteristic polynomials. Even in the case of integral orthogonal groups, a clean criterion necessary for this approach to work
does not seem to be known.

On the other hand, the upper bound for the conjugacy-class count in Theorem \ref{strong} is a special case of Theorem \ref{up}, which applies to any lattice in a semisimple real algebraic group. The proof proceeds by relating the Jordan projection to the Cartan projection and by applying the standard orbital-counting technique of Eskin-McMullen \cite{EM}, which exploits the mixing of the $G$-action on $\Ga\ba G$ and the strong wavefront lemma (\cite[Theorem 3.7]{GO}).

\medskip 

We conclude the introduction by formulating the following conjecture:
\begin{Con}\label{cj}
Let 
$\Gamma$ be an arithmetic lattice of a connected simple real algebraic group $G$.
Fix a positive Weyl chamber $\fa^+\subset \fa$ and let $\rho_G$ be the half-sum of all positive roots of $(\frak g, \frak a^+)$, where $\frak g=\op{Lie} G$.
For  $v\in\operatorname{int}\fa^{+}$, define  directional entropies
$\mathsf E_{\Gamma}(v)$ and
 $\mathsf E_{\Gamma}^{\star}(v)$ as in Definition \ref{dir2}.
Then
$$
  \mathsf E_{\Gamma}(v)=\rho_{G}(v) \quad\text{and }\quad  \mathsf E_{\Gamma}^{\star}(v)=2\,\rho_{G}(v).
$$
\end{Con}

If $G$ has  rank-one,
 the prime geodesic theorem for rank-one
locally symmetric manifolds (see, for instance, \cite{Se}, \cite{He}, \cite{Mat}, \cite{GW}, 
\cite{Ro}, \cite{MMO}, etc.) implies that $\mathsf E_{\Gamma}^{\star}(v)=2\rho_G(v)$.
As we shall see in Theorem~\ref{cartan}, the corresponding entropy $\mathsf E^\star_\Ga(v)$ defined
via the Cartan projection is always $2\rho_G(v)$; it seems plausible that the
Jordan and Cartan counts differ only by a polynomial factor, in which case the second equality in the above
would indeed hold.

\subsection*{ Acknowledgements}
I am grateful to Curt McMullen for many stimulating discussions and to Akshay Venkatesh for suggesting the polynomial-counting viewpoint that proved decisive in our entropy arguments. I also thank Emmanuel Breuillard, Sebastian Hurtado,  Dongryul Kim and Arul Shankar for helpful comments on this work.

\section{Root separations and Proof of Theorem \ref{ms}}
Let $n\ge 2$. As $T\to \infty$, the number of monic integral polynomials of degree $n$ whose
roots are bounded by
$e^T$ grows in the order of $e^{n(n+1)T/2}$. If we additionally require
the constant term to be $\pm 1$, the growth rate drops to
the order $e^{n(n-1)T/2}$. These orders remain unchanged when we
restrict to totally real polynomials \cite{AP}.

In this section, we fix a vector $v\in \mathsf H_+$ and a sign pattern $$m=(m_1, \cdots,  m_n) \in \{\pm 1\}^n,$$
and  count those polynomials
whose roots lie near the prescribed points
$$m_ie^{Tv_i}\quad 1\le i\le n,$$
up to an additive error order
$O(\e e^{Tv_i}) $ for a fixed $\e>0$. The proof relies on translating the information about the roots into precise size constraints on the polynomial's coefficients.

\begin{Def} For $\e>0$ and  $T>1$, denote by
$$\mathscr {Q}_{T}(v,m; \e) \quad\text{ (resp.  $\mathscr {Q}^{\irr}_{T}(v,m; \e) $)}$$ the set of all monic integral (resp. irreducible\footnote{Throughout the paper, irreducible means irreducible over $\z$}) polynomials
   with roots  $x_1, \cdots, x_n$ such that
   \be\label{xxxx} |x_i-m_ie^{Tv_i}|\le  \e e^{Tv_i} \quad\quad \text{for all }i=1,\cdots, n.\ee \end{Def} 

    Set \be\label{dv} \delta_v:=\min_{1\le i\le n-1} (v_1+\cdots + v_i ) \ee

 \begin{thm}\label{qq}  Let $\e>0$. As $T\to \infty$,
$$   \#  \mathscr {Q}_{T}(v,m; \e) \asymp_\e 
\exp { \tfrac{1}{2} \sum_{i<j} (v_i-v_j ) T } .$$
More precisely, there exist absolute constants $c_1, c_2>0$ such that for all $T$ large enough depending on $n$ and $\e$,
\begin{multline*}  \left(\tfrac{2\e}{ (n-1) 3^n} \right)^{n-1} e^{ \tfrac{1}{2} \sum_{i<j} (v_i-v_j ) T }   \left( 1 -
c_1 \tfrac{n}{\e}  e^{-\delta_v T}  \right) \le   \#  \mathscr {Q}_{T}(v,m; \e) 
\\ \le     (2\e)^{n-1} \, n!\, 
e^{ \tfrac{1}{2} \sum_{i<j} (v_i-v_j ) T }   \left( 1+ 
c_2 \tfrac{n}{\e} e^{-\delta_v T}  \right)  \end{multline*}


and 
\begin{multline*}  \left(\tfrac{2\e}{ (n-1) 3^n} \right)^{n-1} e^{ \tfrac{1}{2} \sum_{i<j} (v_i-v_j ) T }   \left( 1 -
c_1 \tfrac{2^n}{\e}  e^{-\min(\delta_v, \eta_v) T}  \right) \le   \#  \mathscr {Q}^{\irr}_{T}(v,m; \e) 
\\ \le     (2\e)^{n-1} \, n!\, 
e^{ \tfrac{1}{2} \sum_{i<j} (v_i-v_j ) T }   \left( 1+ 
c_2 \tfrac{2^n}{\e} e^{-\min(\delta_v,\eta_v)  T}  \right)  \end{multline*}

where $\eta_v>0$ is defined in \eqref{ev}.
  \end{thm}

This theorem follows from three lemmas \ref{r6}, \ref{r2} and \ref{r7} below.
   
To motivate  Definition \ref{pt}, let us first examine the size of each coefficient of the following reference polynomial
\be\label{qm} q_{T} (x)=q_{Tv, m}(x) \defeq \prod_{i=1}^n (x- m_i e^{Tv_i }).\ee 
Writing $q_{T}(x)=\sum_{k=0}^{n}(-1)^{\,n-k}b_{\,n-k}x^{k} = x^{n}-b_{1}x^{\,n-1}+b_{2}x^{\,n-2}
  -\cdots+(-1)^{n}b_{n},$ Vieta's formulas give:
$$
  b_{i}=
  \sum_{\substack{S\subset\{1,\ldots,n\}\\|S|=i}}
      \Bigl(\prod_{j\in S}m_{j}\Bigr)
      e^{T\sum_{j\in S}v_{j}} \quad\text{$1\le i\le n$.}
$$

Therefore, for any $0<\e<1$, there exists $T_1=T_1(v,\e)>0$  such that for all $T\ge T_1$ and
all  $1\le i\le n$, 
$$  (1-\e)   e^{T(v_1+\cdots +v_i) }\; \le\; b_{i}M_i \; \le \; (1+\e) e^{T(v_1+\cdots +v_i) }$$ 
where $M_i= \prod_{j=1}^i m_j$.

\begin{Def}\label{pt}
For $0<\e<1$, let $\mathscr {P}_{T}(v,m; \e)$ be the set  of all monic integral
 polynomials 
 $$p(x)=\sum_{i=0}^n  (-1)^{n-i} a_{n-i} x^i $$
such that  for all $1\le i\le n$, 
\be\label{pp}   (1-\e)   e^{T(v_1+\cdots +v_i) } \le a_{i}M_i \le (1+\e) e^{T(v_1+\cdots +v_i) }.\ee 

We also define $\mathscr {P}'_{T}(v,m; \e)$ to be the set  of all monic integral
 polynomials 
 $p(x)=\sum_{i=0}^n  (-1)^{n-i} a_{n-i} x^i $ such that 
 for all $1\le i\le n$, 
\be\label{pprime}   (1-(i+1) \e)   e^{T(v_1+\cdots +v_i) } \le a_{i}M_i \le (1+(i+1)\e) e^{T(v_1+\cdots +v_i) }.\ee

\end{Def}
Note that for all  $0<\e<1/(n+1)$, any $p\in \mathscr {P}'_{T}(v,m; \e)$ satisfies $a_n=M_n$.
Throughout the paper, we repeatedly use the following simple identity:
 \be\label{id} \sum_{i=1}^{n-1} (v_1+\cdots +v_i)= \tfrac{1}{2} \sum_{1\le i<j\le n} (v_i-v_j )=\sum_{i=1}^{n-1} (n-i)v_i\ee 
 where $\sum_{i=1}^n v_i=0$ is used.

 We immediately get the following by counting integral vectors in the axis-parallel boxes given by \eqref{pp} and
 \eqref{pprime} from the classical theorem of Davenport \cite{Da}.
Recall the constant $\delta_v>0$ from \eqref{dv}.
\begin{lem}\label{r6}  For any $ 0<\e<\frac{1}{n+1}$ and $T>1$ large enough, we have
we have 
\begin{multline*}   (2\e)^{n-1} 
e^{ \tfrac{1}{2} \sum_{i<j} (v_i-v_j ) T }  \left( 1- \tfrac{n}{\e} e^{-\delta_v T} \right) \le  \#  \mathscr {P}_{T}(v,m; \e) \\
\le  (2\e)^{n-1} 
e^{ \tfrac{1}{2} \sum_{i<j} (v_i-v_j ) T } \left( 1+ \tfrac{n}{\e} e^{-\delta_v T}\right) \end{multline*} 
and 
\begin{multline*}   (2\e)^{n-1} n!
e^{ \tfrac{1}{2} \sum_{i<j} (v_i-v_j ) T }  \left( 1- \tfrac{n}{\e} e^{-\delta_v T} \right) \le  \#  \mathscr {P}'_{T}(v,m; \e) \\
\le  (2\e)^{n-1} n!
e^{ \tfrac{1}{2} \sum_{i<j} (v_i-v_j ) T } \left( 1+ \tfrac{n}{\e} e^{-\delta_v T}\right) \end{multline*}

\end{lem}
\begin{proof}
Davenport's theorem \cite{Da} gives that for any axix-parallel box $B=\prod_{i=1}^d [a_i, b_i] \subset \br^d$,
\be\label{da} |\#(B\cap \z^d)-\text{Vol}(B)|\le  \sum_{i=1}^d \prod_{j\ne i} (b_j-a_j)=\sum_{i=1}^{d} \frac{\text{vol}(B)}{b_i-a_i}.\ee  Setting $E_i=e^{(v_1+\cdots +v_i)T}$ and $B=\prod_{i=1}^{n-1} [(1-\e )E_i, (1+\e) E_i]$, 
  the number $ \#  \mathscr {P}_{T}(v,m; \e)  $
  is same as $\# (\z^{n-1}\cap B)$ and hence the first claim follows by \eqref{id} and \eqref{da}. The second claim follows similarly.
\end{proof}

\begin{lem}[Root approximation] \label{r2} Let $c_n= (n-1) 3^{n}$.
For any $0<\e <1/(4c_n)$, there exists 
$T_0= T_0(v, \e)\ge 1$ such that for all $T\ge T_0$,
   every polynomial $p\in \mathscr {P}_{T}(v,m; \e) $ has $n$-distinct {\it real} roots $x_1, \cdots, x_n$ with 
   \be\label{xx} |x_i-m_ie^{Tv_i}|\le c_n \e e^{Tv_i} \quad\quad \text{for all }i=1,\cdots, n.\ee 
 
   Conversely,  any monic polynomial $p\in \z[x]$ with roots $x_1, \cdots, x_n$ satisfying \eqref{xx}
   belongs to  $\mathscr {P}'_{T}(v,m; \e)$. 

  In other words, for all $T$ sufficiently large,
   $$\mathscr {P}_{T}(v,m; \frac{\e}{c_n} )\subset \mathscr {Q}_{T}(v,m; \e ) \subset \mathscr {P}'_{T}(v,m;  \e).$$
\end{lem}

\begin{proof} The second statement is a simple consequence of Vieta's formulas. 
Let $w_{i}=\sum_{j=1}^{i} v_j$ for each $1\le i\le n-1$. Let $q_T$
be as in \eqref{qm} so that $q_T(x)=\sum_{i=0}^n (-1)^{n-i} b_{n-i}x^i$.
So for all $T\ge T_1(v, \e)$ and
for all $1\le i\le n-1$, 
$$  e^{Tw_{i} -\e} \le b_{i}M_i\le e^{Tw_{i}+\e}.$$
By increasing $T_1$ if necessary, we may assume that $e^{Tv_i}\ge 3 e^{Tv_{i+1}}$ and
$e^{Tw_i}\ge 3 e^{Tw_{i+1}}$ for all $i$. Consequently
\(
   |e^{T v_i}-e^{T v_{i+1}}|\ge\bigl(\tfrac23\bigr)e^{T v_i}.
\)

Fix $0<\e < (4c_n)^{-1}$. Then we have  
\be\label{cc} 3(n-1)(1+c_n\e)^{n-1} < c_n(2/3-c_n\e)^{n-1}.\ee 
To check this, we note that 
$f(\epsilon)=  \frac{3(n-1)(1+c_n\e)^{n-1} }{ c_n(2/3-c_n\e)^{n-1}}$ is a strictly increasing function on the interval $(0, (4c_n)^{-1})$ and $f((4c_n)^{-1})=1$.

For each $1\le j\le n$, consider the discs
$$D_j =\{ x\in \mathbb C: |x- m_j e^{Tv_j}| \le  c_n \e  e^{Tv_j}\}.$$
Since $c_n \e \le 1/4$, we have $(1+c_n\e) e^{Tv_{j+1} } < (1-c_n\e) e^{Tv_j}$ for all $j$ and hence these discs are pairwise disjoint.

Let  $p_{T}(x) =\sum_{i=0}^n  (-1)^{n-i} a_{n-i} x^i $ be a polynomial in $\mathscr {P}_{T}(v,m; \e)$. We claim that for all $T$ sufficiently large,  $p_T(x)$ has precisely one root inside each disc $D_j$.
Write
$$\Delta_{T}(x)\coloneqq q_{T}(x)-p_{T}(x) =\sum_{i=1}^{n-1} (-1)^{n-i} (b_{n-i}-a_{n-i})  x^i.$$
From~\eqref{pp},
$\lvert b_{\,n-i}-a_{\,n-i}\rvert\le 3\e\,e^{T w_{n-i}}$.
Hence for all $x\in\partial D_{j}$, we have
\begin{equation}\label{dt}
  \lvert\Delta_{T}(x)\rvert
     \le 3\e\bigl(1+c_{n}\e\bigr)^{\,n-1}
          \sum_{i=1}^{n-1}e^{T w_{n-i}+i T v_{j}}.
\end{equation}
On the other hand,
$$|q_T(x) |=\prod_{i=0}^n |x-m_i e^{Tv_i }| = c_n\e e^{Tv_j} \prod_{i<j} |x-m_ie^{Tv_i }|\cdot \prod_{i>j} |x-m_ie^{Tv_i }| .$$
For $i<j$,
$$ |x-m_i e^{Tv_i }|\ge |m_ie^{Tv_i}-m_j e^{Tv_j}| -|x-m_je^{Tv_j}|\ge \frac{2}{3} e^{Tv_i} - c_n\e e^{Tv_j} \ge (\frac{2}{3} -c_n\e) e^{Tv_i} .$$
For $i>j$, we similarly have
$$ |x-m_i e^{Tv_i }|\ge (\frac{2}{3} -c_n\e) e^{Tv_j} .$$
Hence
\be\label{qt} |q_T(x) | \ge c_n\e \left( \frac{2}{3} -c_n\e \right)^{n-1} e^{T (\sum_{i<j} v_i) + (n-j+1) T v_j}.\ee 
Let $$S(T, j)=\sum_{i=1}^{n-1}
e^{ T w_{n-i} + i Tv_j} \; \;\text{and} \;\;  R(T, j)=  e^{T (\sum_{k<j} v_k) + (n-j+1) T v_j}.$$
Let $$\Delta_{i, j}=\left(\sum_{k<j} v_k +(n-j+1)v_j \right) -(v_1+\cdots+ v_{n-i} + i v_j)$$ so that  \be\label{st} \frac{S(T, j)}{R(T, j)} =\sum_{i=1}^{n-1} e^{-T \Delta_{i, j}}. \ee 

We check that $\Delta_{i, j}\ge 0$ by writing 
\be \Delta_{i, j}= \begin{cases}  (n-i -(j-1)) v_j - (v_j+\cdots +v_{n-i})\ge 0 & \text{ if $n-i\ge j-1$,}\\ 
 (v_{n-i+1} +\cdots +v_{j-1}) - ((j-1)-(n-i)) v_j  \ge 0 & \text{ otherwise}.\end{cases}\ee 
Hence by \eqref{st},
we have $$ \frac{S(T, j)}{R(T, j)} \le n-1. $$  
Therefore by \eqref{dt} and \eqref{qt}, and since $\e$ satisfies \eqref{cc},
we get that for all $x\in \partial D_j$, 
\begin{multline*}
    |\Delta_T (x)| \le  3\e (1+c_n\e)^{n-1} S(T, j) \le  (n-1) 3\e (1+c_n\e)^{n-1} R(T, j)
\\
< c_n\e (2/3-c_n\e)^{n-1} R(T, j) \le |q_T(x)|,
\end{multline*}
and hence
$$|\Delta_T(x)| < |q_T(x)| .$$
Hence by Rouch\'e's theorem (cf. \cite{Ah}), two polynomials $q_T(x)$ and $p_T(x)$ have the same number of zeros (counted with  multiplicity) inside each $D_j$.
Since $D_j$ are pairwise disjoint and  $q_T$ has exactly one root in each $D_j$, the same holds for $p_T$. Since $p_T$ has real coefficients
and each $D_j$ is invariant under complex conjugation, 
$p_T(x)$ has one {\it real} root $x_j$ such that 
$$|x_j-m_je^{Tv_j}| \le c_n \e e^{Tv_j}.$$
Hence $p_T\in   \mathscr {Q}_{T}(v,m; c_n\e )$.
This finishes the proof.
\end{proof}

Denote by $\op{Disc}(p)=\prod_{i\ne j}(x_i-x_j)=\prod_{i<j} (x_i-x_j)^2$
the discriminant of a polynomial $p$ with roots $x_1, \cdots, x_n$. 
For the polynomial $q_{Tv, m}(x)=\prod_{i=1}^n (x-m_i e^{Tv_i})$,
its discriminant $\op{Disc} (q_{Tv, m})$ satisfies
$$ {\op{Disc} (q_{Tv,m})} =  e^{ \left(\sum_{1\le i<j\le n} v_i \right) T} \left( 1+O(e^{-\eta T})\right)\quad\text{ for some $\eta>0$} $$
and hence
$$\lim_{T\to\infty}\tfrac{1}{T} \log {\op{Disc} (q_{Tv,m})} = \sum_{i<j} (v_i-v_j ).$$

The following is a simple consequence of Lemma \ref{r2}:
\begin{cor}
For all small $\e >0$, there exist
$T_0= T_0(v, \e)>0$ such that for all $T\ge T_0$,
   every polynomial $p_T\in \mathscr {P}_{T}(v,m; \e)$ satisfies
   $$(1-C_n\e)  e^{\sum_{i<j} (v_i-v_j) T} \le \op{Disc}(p_T) \le (1+C_n \e) e^{\sum_{i<j} (v_i-v_j)T}$$
   where $C_n=n(n-1)/2$.
   In particular, $$\lim_{T\to\infty}\tfrac{1}{T} \log {\op{Disc} (p_{T})}= \sum_{i<j} (v_i-v_j ) .$$
\end{cor}

We will need the following estimates in the next lemma \ref{r7}:
\begin{lemma}\label{cal}
Let  $\{1,\dots,n\}=S_1\sqcup S_2$ be a partition into two non-empty subsets so that
  $\displaystyle\sum_{i\in S_j}v_i=0$ for $j=1,2$.  Writing
  $S_1=\{i_1<\dots<i_{\ell_1}\}$ and
  $S_2=\{j_1<\dots<j_{\ell_2}\}$ with $\ell_1+\ell_2=n$, we have
$$
     \sum_{k=1}^{\ell_1-1}(v_{i_1}+\cdots+v_{i_k})
     \;+\;
     \sum_{k=1}^{\ell_2-1}(v_{j_1}+\cdots+v_{j_k})
     \;<\;
     \sum_{k=1}^{n-1}(v_1+\cdots+v_k)\;. $$
\end{lemma}

\begin{proof}
We first rewrite the right hand side $\text{(RHS)}$ as
  $$
     \sum_{k=1}^{n-1}(v_1+\dots+v_k)
     =\sum_{k=1}^{n-1}\sum_{i=1}^{k}v_i
     =\sum_{i=1}^{n-1}(n-i)\,v_i
     =\sum_{1\le i<j\le n}v_i.
  $$
Similarly, the left hand side, for each $j=1,2$,
  $$
     \sum_{k=1}^{\ell_j-1}(v_{i_1}+\dots+v_{i_k})
     =\sum_{{i<j, i,j\in S_j}}v_i,
  $$
  and hence
  $$
     \text{LHS}
     =\sum_{{i<j, i,j\text{ in the same }S_\ast}}v_i.
  $$

Set
\be\label{SS} D_v(S_1, S_2):= \sum_{k=1}^{n-1}(v_1+\cdots+v_k) -\left( \sum_{k=1}^{\ell_1-1}(v_{i_1}+\cdots+v_{i_k})
     \;+\;
     \sum_{k=1}^{\ell_2-1}(v_{j_1}+\cdots+v_{j_k})\right) .\ee

Hence 
  $$ 
     D_v(S_1, S_2)
     =\text{(RHS)}-\text{(LHS)}
     =\sum_{{i<j,  S(i)\neq S(j)}}v_i,$$
  where $S(i)$ denotes the block containing $i$.
  Add the same pairs with the
  complementary index:
  $$
     \sum_{{i<j, S(i)\neq S(j)}}(v_i+v_j)
     =|S_2|\sum_{i\in S_1}v_i \;+\;|S_1|\sum_{j\in S_2}v_j
     =0,
  $$
  because each block has total sum $0$. Hence $D_v(S_1, S_2)=   -\sum_{{i<j,  S(i)\neq S(j)}}v_j$; so
  $$
     2D_v(S_2, S_2)
     =\sum_{{i<j, S(i)\neq S(j)}}(v_i-v_j).
  $$

  Since
  $v_1>\dots>v_n$, every difference $v_i-v_j\;(i<j)$ is
  strictly positive.  There is at least one cross pair (the blocks are
  non-empty), so $D_v(S_1, S_2) >0$.
\end{proof}

Define  \be\label{ev} \eta_v:=\min D_v(S_1, S_2) >0\ee 
where $D_v(S_1, S_2)$ is as in \eqref{SS} and
the minimum is taken over all non-trivial partitions of $\{1,\cdots, n\} =S_1\sqcup S_2$.
The function $v\mapsto \eta_v$ is clearly continuous on $\mathsf H_+$ and hence
$\min_{v\in Q} \eta_v>0$ for any compact subset $Q\subset \mathsf H_+$.
\begin{lem} \label{r7}  Let $0< \e<1$. As $T\to \infty$, the proportion of irreducible polynomials in $\mathscr {P}_{T}(v,m; \e) $ tends to $1$ exponentially fast:  there exists $T_0=T_0(n, \e)$ such that for all $T\ge T_0$,
$$
{\# \{p\in \mathscr {P}_{T}(v,m; \e)\text{ irreducible}\} }=
{\# \mathscr {P}_{T}(v,m; \e)}\cdot \left( 1 +O(2^{n} \e^{-1} e^{-\eta_v T})\right) $$
where the implied constant is an absolute constant. 
The same type of estimate holds for $\mathscr {P}'_{T}(v,m; \e)$.
\end{lem}

\begin{proof} Let $T\ge T_0(v,\e) $ be as in  Lemma \ref{r2}. 
For any $p \in \mathscr {P}_{T}(v,m; \e)$,
 by Lemma \ref{r2},
$p$ has distinct roots $x_1, \cdots, x_n$ such that
 $|x_i-m_ie^{Tv_i}|\le c_n \e e^{Tv_i}$ for each $1\le i\le n$.
Suppose that $p \in \mathscr {P}_{T}(v,m; \e)$ is reducible over $\z$.  
We then  have a partition of $\{1, \cdots, n\}$ as the disjoint union  $S_1\sqcup S_2$ of non-empty subsets  such that $p(x)= f_1(x) f_2(x)$
where $f_j(x)=\prod_{k \in S_j} (x- x_k)\in \mathbb Z[x]$ for $j=1,2$.  
List elements of $S_j$
as ${j_1}> {j_2}>\cdots >{j_{\ell_j}}$. Let $u_j= (v_{j_1}, \cdots, v_{j_{\ell_j}})$ and 
$M'_j= (m_{j_1}, \cdots, m_{j_{\ell_j}}) $.
  It follows that
 $$f_j\in \mathscr P_T(u_j,M'_j; \e) .$$ Hence for all sufficiently large $T$,
by Lemma \ref{r6}, for all sufficiently large $T\ge 1$, we have 
$$ \# \mathscr P_{T}(u_j, M_j'; \e) \le  2 (2\e)^{\ell_j-1}  e^{T \sum_{k=1}^{\ell_j-1} (v_{j_1}+\cdots + v_{j_k}) } .$$ 
Since the constant terms of $f_1$ and $f_2$ are $\pm 1$, we have  $\sum_{i\in S_j} v_i = 0$. 
By Lemma \ref{cal}, we have
$$ 
\sum_{k=1}^{\ell_1-1} (v_{1_1}+\cdots + v_{1_k}) +\sum_{k=1}^{\ell_2-1} (v_{2_1}+\cdots + v_{2_k})
 <  \sum_{k=1}^{n-1} (v_1+\cdots +v_k) , $$
 that is, $D_v(S_1, S_2)>0$.
 Therefore
 we have $$ \frac{ \# \mathscr P_{T} (u_1,M_1'; \e)  \cdot \# 
 \mathscr P_T (u_2,M_2'; \e)  }{ \#\mathscr P_{T}(v,m ; \e) }\le 8 \e^{-1} e^{-D_v(S_1, S_2) T}. $$

Since this holds for any non-trivial partition of $\{1, \cdots, n\}$ into two non-empty subsets and there are at most $2^{n-1}$ number of such
partitions,
this proves the first claim by the definition of $\eta_v$. The proof for $\mathscr P'_T(v,m;\e)$
is similar.
\end{proof}

\begin{proof}[Proof of Theorem \ref{qq}]
    The first claim follows from Lemma \ref{r6} and Lemma \ref{r2}.
    By Lemma \ref{r2} and Lemma \ref{r7}, for all $T$ sufficiently large, the cardinality of the
    set of {\it reducible} polynomials in $\mathscr {Q}_{T}(v,m; \e)$ is at most
    $ c 2^n \e^{-1}   e^{-\eta_v T} \cdot \# \mathscr {Q}_{T}(v,m; \e)$ for some absolute
    constant $c>0$. Hence the second claim follows from this and the first claim.
\end{proof}

\section{ Directional entropy of totally real algebraic units} We now apply our
polynomial analysis to compute the directional entropy for totally real units of degree $n$.
Fix $v\in \mathsf H_+$ and a sign pattern $m=(m_1, \cdots, m_n)\in \{\pm 1\}^n$.
We use the notation $\cal K_n, O_n^\times, \Sigma_K$, etc. from the introduction.
For simplicity, for $\u\in (O_K^\times, \sigma)$,
we write $\|\sigma(\u)- me^{Tv}\|\le \e e^{Tv}$ to mean 
that $ | \sigma_i( \u)- m_i e^{Tv_i} |<\e e^{Tv_i}
   \text{ $\forall i$} $. 
  For each $T>1$, define $$\mathscr U_{T}(v, m;\e) =\bigcup_{(K,\sigma)\in \cal K_n}
   \bigl\{ \u\in  (O_K^\times, \sigma) :
\|\sigma(\u)- me^{Tv}\|\le \e e^{Tv}  \bigr\}.
   $$

 Let $\mathscr U_{T}^{\op{prim}}(v, m;\e)$ be the set of all $\u\in (O_K^\times, \sigma)$,
$(K, \sigma)\in \cal K_n$, such that the field
$\mathbb Q(\u)$ has degree $n$, or equivalently, $p(x) =\prod_{i=1}^n (x-\sigma_i(\u)) $ is irreducible
over $\z$ for $\sigma=(\sigma_1, \cdots, \sigma_n)$.
   \begin{lem} \label{prim} Let $0\le \e <1/2$. Then for all sufficiently large $T>1$,
   we have
      $$ \mathscr U_{T}(v, m;\e)= \mathscr U_{T}^{\op{prim}}(v, m;\e).$$
   \end{lem}

   \begin{proof}
       If $\u \in \mathscr U_{T}(v, m;\e) \cap (O_K^\times, \sigma)$ is non-primitive, 
       then there is a subfield $K_0$ of $K$ of degree $1<m<n$ such that $\u\in O_{K_0}^\times$ and each of the $m$-embeddings $K_0\hookrightarrow \br$ extends to precisely $n/m$-embeddings to $K$ into $\br$. Therefore for some $i< j$,
       $\sigma_i(\u)=\sigma_j(\u)$. This implies that     $|m_i e^{Tv_i} -m_j e^{Tv_j}| \le \e e^{Tv_i}$.
      Since  $v_i>v_j$ and hence $|m_i e^{Tv_i} -m_j e^{Tv_j}|= e^{Tv_i} (1 \pm e^{T(v_j-v_j)}) \ge  e^{Tv_i}/2 $ for all sufficiently large $T$, $\u$ cannot be non-primitive if $T$ is large enough.
   \end{proof}
   
   \begin{prop} \label{um} 
  Let $v\in \mathsf H_+$.  For all small $\e>0$, we have
   $$   \left(\tfrac{2\e}{ (n-1)3^n} \right)^{n-1}  \le \liminf_{T\to \infty}  \frac{\#  \mathscr {U}_{T}(v,m; \e)}{e^{ \tfrac{1}{2} \sum_{i<j} (v_i-v_j ) T }} \le \limsup_{T\to \infty} \frac{\#  \mathscr {U}_{T}(v,m; \e)}{e^{ \tfrac{1}{2} \sum_{i<j} (v_i-v_j ) T }} \le  (2\e)^{n-1} \, n!\,  . $$
More precisely, there exist absolute constants $c_1, c_2>0$ such that for all large $T>1$ depending on $n$ and $\e$, we have
   \begin{multline*}  \left(\tfrac{2\e}{ (n-1) 3^n} \right)^{n-1} e^{ \tfrac{1}{2} \sum_{i<j} (v_i-v_j ) T }   \left( 1 -
c_1 2^n \e^{-1} e^{-\min(\delta_v, \eta_v) T}  \right) \le    \#  \mathscr {U}_{T}(v,m; \e)
\\ \le     (2\e)^{n-1} \, n!\, 
e^{ \tfrac{1}{2} \sum_{i<j} (v_i-v_j ) T }   \left( 1+ 
c_2 2^n \e^{-1} e^{- \min(\delta_v, \eta_v)  T}  \right)  \end{multline*}
 \end{prop}
where $v\mapsto \delta_v$ and $v\mapsto \eta_v$ are positive continuous functions 
given in \eqref{dv} and \eqref{ev} respectively.
\begin{proof}

Let $\e>0$ and $T\ge T_0(v,\e)$ as in Lemma \ref{r2}. 
Let $p \in \mathscr {Q}^{\irr}_{T}(v,m; \e)$, and let $K$ be its splitting field, which must be a totally real number field of degree $n$. Let $x_1, \cdots, x_n$ be the roots of $p$ ordered so that 
 $|x_1|>\cdots>|x_n|$. Since $x_i\in O_K$ and $\prod_{i=1}^n x_i=\prod_{i=1}^n m_i=\pm 1$, there exists a unit $\u\in O_K^\times$ and $\sigma=(\sigma_1, \cdots, \sigma_n) \in \Sigma_K$ such that $K=\q (u)$ and $x_i=\sigma_i(\u)$.
Hence $$\# \mathscr {Q}^{\irr}_{T}(v,m; \e) \le \# \mathscr U_{T}^{\op{prim}} (v, m; \e) .$$
 
 Conversely, let $\u\in \mathscr U_{T}^{\op{prim}} (v, m;\e) \cap (O_K^\times, \sigma)$.
Setting $p (x)=\prod (x-\sigma_i(\u))$, we have
$p \in \mathscr {Q}^{\irr}_{T}(v,m; O(\e)) .$ 
Moreover, this map is injective for $T$ large enough. To see this, suppose
that there exist $\u \in \mathscr U_{T}(v, m;\e)^{\op{prim}}\cap  (O_K^\times, \sigma)$ and $\u'\in \mathscr U_{T}(v, m;\e)^{\op{prim}}\cap  O_{K'}^\times, \sigma')$ such that
$p(x)=q(x)$ where $p(x) =\prod (x-\sigma_i(\u)) $ and $q(x) =\prod(x-\sigma_i'(\mathtt u'))$.
Since $K$ and $K'$ must be the splitting fields of $p$ and $q$ respectively, $K=K'$
and $\{\sigma_i'(\u'): i=1, \cdots, n\} =\{\sigma_i(\u): i=1, \cdots, n\}$.
Since the intervals $(Tv_i-\e, Tv_i+\e)$ are pairwise disjoint once $T$ is sufficiently big, and
$\log \sigma_i(\u), \log \sigma'_i(\u') \in (Tv_i-\e, Tv_i+\e) $ for all $T$ sufficiently large, 
we must have $\sigma_i(\u)=\sigma_i'(\u')$ for all $1\le i\le n$.
Hence for all sufficiently large $T\gg 1$, 
$$\# \mathscr U_{T}^{\op{prim}} (v, m;\e) \le  \# \mathscr {Q}_{T}^{\irr} (v,m; \e).$$

Therefore the claim follows from Lemma \ref{prim} and  Theorem \ref{qq}.
\end{proof}

Observe that once $T$ is sufficiently large depending only on $v$ and $\e$,
the sets $\mathscr U_{T}(v, m;\e)$, $m\in \{\pm 1\}^{n}$, are pairwise disjoint.
Since
$$\{u\in O_n^\times: \|\Lambda(u)-Tv\|\le \e\} =
\bigsqcup_{m\in \{\pm 1\}^{n}} \mathscr U_{T}(v, m;\e),$$ Theorem \ref{ms} follows from  Proposition \ref{um}.

   Define
$$
\mathsf E_{n} (v, m)= \lim_{\e \to 0} \lim_{T\to \infty}
\frac{1}{T} \log \#\mathscr U_{T}(v, m;\e)  ,
$$ if the limit exists. As an immediate consequence of Proposition \ref{um}, we have:
\begin{thm}\label{iup1}
We have
    $$
    \mathsf E_{n} (v,m) =\tfrac{1}{2} \sum_{i<j} (v_i-v_j ).
    $$ \end{thm}

\section{Eigenvalue entropy of $\SL_n(\z)$}
In this section, we count the eigenvalue patterns of $\SL_n(\z)$ that lie
in a thin tube around a fixed ray, invoking Theorem \ref{qq}.

Fix $v\in \mathsf H_+$ and a sign pattern $m=(m_1, \cdots,  m_n) \in \{\pm 1\}^n$ with $\prod_{i=1}^n m_i=1$.
Let 
$$\rho_{\SL_n} (v)=\frac{1}{2} \sum_{i<j}(v_i-v_j);$$
be the half-sum of all positive roots of $\SL_n(\br)$.
For each loxodromic element  $g\in\SL_n(\mathbb R)$, let $\cal E(g)$, $\lambda(g)$ and $m(g)$
be its eigenvalue pattern, Jordan projection and sign pattern as defined in \eqref{eg} and \eqref{eg2}. 
Set
$$ J_{ T}(v,m; \e):= \#\Bigl\{ (\lambda(\ga), m(\ga)) :\ga\in\SL_n(\mathbb Z),
              \|\lambda(\ga)-Tv\|_{\max} \le\e, m(\ga)=m\Bigr\}. $$

       \begin{thm} \label{jm} For all small $\e >0$, we have 
       $$     \left(\tfrac{2\e}{ (n-1)3^n} \right)^{n-1} \le \liminf_{T\to \infty}  \frac{\#   {J}_{T}(v,m; \e)}{e^{ \rho_{\SL_n}(v)  T }} \le \limsup_{T\to \infty} \frac{\#   {J}_{T}(v,m; \e)}{ e^{ \rho_{\SL_n}(v)  T }} \le (2\e )^{n-1} n!  . $$
In particular, \be\label{ee}
\mathsf E_{\SL_n(\z)}(v, m)=\rho_{\SL_n}(v).
\ee 
   \end{thm}

\begin{proof} There exists $T_1=T_1(v, \e)>0$ such that for all $T\ge T_1$ and
for each  $(\la(\ga), m(\ga))\in J_{T}(v,m; \e)$,
the polynomial $p(x)=\prod (x- m_i(\ga) e^{\la_i(\ga)} )$ belongs to  $ \mathscr {Q}_{T}(v,m; \e) $. Since this gives an injective map, 
we have
$\#  J_{ T}(v,m; \e) \le \# \mathscr {Q}_{T}(v,m; \e)$.

Let $f \in \mathscr {Q}_{T}(v,m; \e) =\sum_{i=0}^n (-1)^{n-i} a_{n-i} x^i$.   
Consider the companion matrix of $f$:
$$
C_f=\begin{pmatrix}
 0 & 0 & \cdots & 0 & (-1)^{n+1} a_{n}\\
 1 & 0 & \cdots & 0 & (-1)^n a_{n-1}\\
 0 & 1 & \ddots & \vdots & \vdots\\
 \vdots & \ddots & \ddots & 0 & -a_{2} \\
 0 & \cdots & 0 & 1 & a_1
\end{pmatrix}
$$

Since $\text{det } C_f=a_n=\prod_{i=1}^n m_i=1$,  we have $C_f\in \SL_n(\z)$. 
If $x_1, \cdots, x_n$ are distinct real roots of $f$ ordered so that $|x_1|>\cdots >|x_n|$, then
$$|x_i - m_ie^{Tv_i}|\le \e e^{Tv_i},\quad\quad 1\le i\le n .$$
Therefore $\| \la(C_f) -Tv\|\le \e$ and $m(C_f)=m$.
Hence the assignment $f\mapsto (\lambda(C_f), m(C_f))$ 
gives a  map from the set $\mathscr {Q}_{T}(v,m; \e) $ to $J_{T}(v,m; \e)$.  Since $(\lambda(C_f), m(C_f))$ describes all roots of $f$, this map is also injective.  Therefore $\#  J_{ T}(v,m; \e) \ge \# \mathscr {Q}_{T}(v,m; \e)$.
Hence the claim follows from
Theorem \ref{qq}.
\end{proof}

The lower bound stated below follows directly from Theorem \ref{jm} and the corresponding
upper bound will be proved in Theorem \ref{up} in a more general setting.
\begin{thm} \label{gen2}
For each $v\in\mathsf H_+$, each sign pattern $m\in \{\pm 1\}$ with $\prod_{i=1}^n m_i=1$, and $\e>0$,
there exist $C_1, C_2>0$ such that
$$C_1 e^{\rho_{\SL_n}(v)T}\le    \#\Bigl\{ [\ga] \in [\SL_n(\mathbb Z)] :
              \|\lambda(\ga)-Tv\|\le\e, m(\ga)=m\Bigr\} \le C_2
              e^{2 \rho_{\SL_n}(v)T} .$$
            In particular,  
\be\label{ee33}
\rho_{\SL_n}(v)\le 
    \underline{\mathsf E}^{\!\star}_{\SL_n(\mathbb Z)}(v,m)
    \;\le\;
    \overline{\mathsf E}^{\!\star}_{\SL_n(\mathbb Z)}(v,m)
   \le 2\rho_{\SL_n}(v) .
\ee 
\end{thm}

\begin{cor} \label{euc3} Let $\|\cdot\|$ be a norm on $\br^n$. Let $v=v_{\|\cdot\|}\in \mathsf H_+$ be a unit vector such that $\max_{\|v\|=1} \rho_{\SL_n}(v)=\rho_{\SL_n}(v_{\|\cdot\|})$.
There exists $C>0$ such that that for all $T>1$,
  $$  \#\{ [\ga]\in [\SL_n(\z)]: \|\lambda(\ga)\| <T, \gamma\in \SL_n(\z) \} \ge  C  
     e^{ \rho_{\SL_n}(v_{\|\cdot\|} )T} .  $$   
   \begin{enumerate}
       \item 
 For the Euclidean norm $\|\cdot\|_{\op{Euc}}$, 
     $\rho_{\SL_n}(v_{\|\cdot\|_{\op{Euc}}}) =\sqrt{\frac{{n(n^2-1)}}{12}}$.
\item For the maximum norm $\|\cdot\|_{\max}$,  $\rho_{\SL_n}(v_{\|\cdot\|_{\op{max}}})$
is $\lfloor n^2/4\rfloor$.
        \end{enumerate}  
\end{cor}
\begin{proof}
    Let $N(T):= \#\{\lambda(\ga): \ga\in \SL_n(\z), \|\lambda(\ga)\| <T\}$.
Since 
$$
  N(T)
\ge  \# J_{m, T-\e}(v_{\|\cdot\|}, \e) $$ for any sign pattern $m$ with $\prod m_j=1$,
the desired lower bound for $N(T)$ follows from Theorem \ref{jm}.

Since $2\rho_{\SL_n}(v)= \sum_{k=1}^{n} (n+1-2k)v_k,$ its maximum on the unit sphere
(for the Euclidean norm) is attained in the direction of
$(n+1-2k)_{k=1}^n$. Hence if we write $v_{\|\cdot\|_{\op{Euc}}}=(v_1^*, \cdots, v_n^*)$,
then \be\label{vs} v_k^*=\frac{n+1-2k}{\sqrt{n(n^2-1)/3}},\quad\text{and}\quad \rho_{\SL_n} (v_{\|\cdot\|_{\op{Euc}}} )= \sqrt{\frac{{n(n^2-1)}}{12}}.\ee 
For the maximum norm, $v_{\|\cdot\|_{\max}}$ is given by
$v_k=1$ whenever $n+1-2k >0$ and $v_k=-1$ whenever $n+1-2k <0$ and $v_k=0$ if $2k=n+1$.
Then for $n=2m$, $2\rho_{\SL_n} (\|v\|_{\max})= \sum_{k=1}^m (2m+1-2k)+\sum_{k=m+1}^{2m}
(2m+1-2k) (-1) = 2m^2=n^2/2$, and for $n= 2m+1$,  $2\rho_{\SL_n} (\|v\|_{\max})=m(m+1)=(n^2-1)/2$.
\end{proof}


\section{Reciprocal polynomials and Counting for $\Sp_{2n}(\z)$}\label{sp}
In this section, we investigate directional entropies for the symplectic lattice $\Sp_{2n}(\z)$.
Our estimates rely on the anaylsis of reciprocal polynomials.

\subsection*{Reciprocal polynomials} A monic polynomial $p\in \br[x]$ of degree $2n$ is called {\it reciprocal} (also called {\it palindromic}) if $$p(x)=x^{2n}p(x^{-1}). $$

Equivalently,
$$       p(x)=\sum_{k=0}^{2n}(-1)^{2n-k}a_{2n-k}\,x^{k},
        \qquad
        a_0=a_{2n}=1,\quad
        a_i=a_{2n-i}\;\;(1\le i\le n);
        $$ or
    $$ p(x)=\prod_{i=1}^{n}\bigl(x-x_i\bigr)\bigl(x-x_i^{-1}\bigr),
        \qquad
        x_1,\dots,x_n\in\mathbb{C} -\{0\}.
$$

Let $$\fa^+=\{v=\text{diag}(v_1, \cdots, v_n, -v_n, \cdots, -v_1): v_1\ge \cdots \ge v_n\ge 0\}.$$
Fix $$v\in \inte\fa^+\quad \text{and}\quad m=(m_1, \cdots , m_n)\in \{\pm 1\}^{n}.$$

\begin{Def} Let $\e>0$ and $T>1$.
Let $\mathscr Q^*_T(v,m;\e)$ (resp. $\mathscr Q^{*,\irr}_T(v,m;\e)$) be the set of all monic integral (resp. irreducible) reciprocal polynomials with roots
$x_1, \cdots, x_n$, $x_1^{-1}, \cdots, x_n^{-1}$ such that for all  $i=1,\cdots, n,$
   $$|x_i-m_ie^{Tv_i}|\le \e e^{Tv_i}, \;\; |x_i^{-1} -m_ie^{-Tv_i}|\le  \e e^{-Tv_i} .$$
\end{Def}
Set  $$\rho^*(v)=\sum_{i=1}^n (v_1+\cdots +v_i)=\sum_{i=1}^n (n+1-i)v_i .$$
The following theorem is a direct combination of Lemma \ref{cr1}, Corollary \ref{r22} and Lemma \ref{r27}:
\begin{thm} \label{qstar} Let $\e>0$.  As $T\to \infty$,
    $$ \# \mathscr Q^*_T(v,m;\e)  \asymp_\e  e^{\rho^*(v)T }  ; $$
    more precisely,
$$   \left(\tfrac{2\e }{(2n-1) 3^{2n}}\right)^n \le \liminf_{T\to \infty}  \frac{\#  \mathscr {Q}_{T}^*(v,m; \e)}{e^{ \rho^*(v)T} } \le \limsup_{T\to \infty} \frac{\#  \mathscr {Q}_{T}^*(v,m; \e)}{e^{ \rho^*(v)T}}  \le(2\e)^n ( n+1)!. $$
Moreover,
      $$ \# \mathscr Q^{*,\irr} _T(v,m;\e)  =  \# \mathscr Q^*_T(v,m;\e) (1+O(e^{-\eta T}))  $$ for some $\eta>0$ depending only on $v$.
\end{thm}

 For $T>1$,
define the model polynomial
$$q_{Tv}(x)=\prod_{i=1}^n (x- m_i e^{Tv_i}) (x-m_i e^{-Tv_i}) =\sum_{k=0}^{2n} (-1)^{2n-k} b_{2n-k} x^{k}.$$ 
Then for any $\e>0$,
and sufficiently large $T\gg 1$, we have that $b_0=1=b_{2n}$,  $b_i=b_{2n-i} $ and
$$ (1-\e) e^{ T(v_1+\cdots +v_i)}   \le b_i  M_i \le  (1+\e) e^{T(v_1+\cdots +v_i)}  \;\;\quad   \text{for all $1\le i\le n$} $$
where $M_i=\prod_{j=1}^i m_i$.

Let $\mathscr P^*_T(v,m;\e)$ be the set of all monic reciprocal polynomials
$$p(x) = \sum_{k=0}^{2n} (-1)^{2n-k} a_{2n-k} x^{k} \in \z[x]$$
such that 
$$ (1-\e) e^{ T(v_1+\cdots +v_i)}   \le a_i  M_i \le  (1+\e) e^{T(v_1+\cdots +v_i)}  \;\;\quad   \text{for all $1\le i\le n$.} $$

Let $\mathscr P^{**}_T(v,m;\e)$ be defined by the condition that 
$$ (1-(i+1)\e) e^{ T(v_1+\cdots +v_i)}   \le a_i  M_i \le  (1+ (i+1) \e) e^{T(v_1+\cdots +v_i)}  \;\;\quad   \text{for all $1\le i\le n$.} $$

Clearly we have:
\begin{lem} \label{cr1} 
    For all sufficiently small $\e>0$, we have, as $T\to \infty$,
$$ \# \mathscr P^*_T(v,m;\e)  \sim (2\e)^n  e^{\rho^*(v)T } \;\; \text{ and } \;\; \# \mathscr P^*_T(v,m;\e)  \sim (2\e)^n (n+1)!  e^{\rho^*(v)T }  .$$
\end{lem}

The following follows from Lemma \ref{r2}:
\begin{cor}[Root approximation] \label{r22} 
For all sufficiently small $\e>0 $, there exists 
$T_0= T_0(v, \e)>1$ such that for all $T\ge T_0$, we have
  $$ \mathscr P^*_T(v,m;\frac{\e}{c_{2n}})\subset \mathscr Q^*_T(v,m;\e)\subset \mathscr P^{**}_T(v,m;  \e) $$
  where $c_{2n}=(2n-1)3^{2n} $.
  \end{cor}

\begin{lem} \label{r27}  For all sufficiently small
$\e>0$, there is $\eta>0$ depending only on $v$ such that for all $T$ large enough, we  have
$$
\frac{\# \{p\in \mathscr P^*_T(v,m;\e) \text{ irreducible}\} }{\# \mathscr P^*_T(v,m;\e)}=1 +O(e^{-\eta T}).$$
\end{lem}

\begin{proof}
Suppose that
$p\in\mathscr P^{*}_{T}(v,m;\varepsilon)$ is \emph{reducible}.
Because $p$ is \emph{reciprocal}, every irreducible factor $f$ of $p$
forces its reciprocal
$f^{*}(x)=x^{\deg f}\,f(x^{-1})$
to be a factor as well.

Consider first those $p$ that factor as $p=f\cdot f^{*}$ with
$f$ irreducible of degree $n$.  Write the roots of $p$ as
$x_i=m_ie^{T v_i}\bigl(1+O(\e)\bigr)$ and $
x_i^{-1}=m_ie^{-T v_i}\bigl(1+O(\e)\bigr)$, $ 1\le i\le n$.
Since the constant term of $f$ must be $\pm 1$, if we set
$P=\{i: f(x_i)=0\}=\{{j_1}< \cdots <{j_\ell}\}$,  then  $1\le \ell <n$.
By Vieta's formula,
the $k$-th coefficient of $f$ is bounded by 
$\exp \left( (v_{j_1} +\cdots + v_{j_k})T\right) $, up to a multiplicative constant. Taking the product over $k=1, \cdots, \ell$,
the  coefficient box for $f$ has volume at most
a constant multiple of
$\exp \sum_{k=1}^{\ell } (v_{j_1} +\cdots + v_{j_k}) $.
Because $\ell<n$ and $v_1>\cdots >v_n>0$,  we have
 $\sum_{k=1}^{\ell } (v_{j_1} +\cdots + v_{j_k}) \le
\sum_{k=1}^n (v_1+\cdots + v_k) -v_n=\rho^*(v)-v_n$.
Hence
\be\label{aa}
\#\{p\text{ with the factorization }p=f\,f^{*}\}
\ll e^{(\rho^{*}(v)-v_n)T}.
\ee

All remaining reducible polynomials split as
$$
p(x)=f_{1}(x)\,f_{2}(x),
\qquad
\deg f_{j}=2s_{j},\;
1\le s_{j}\le n-1,\;
s_{1}+s_{2}=n,
$$
with each $f_{j}$ itself a monic \emph{reciprocal} polynomial in
$\mathbb Z[x]$.

Let
$S\subset\{1,\dots ,n\}$ record which conjugate‐pairs
$\{x_{i},x_{i}^{-1}\}$ of roots of $f_1$ in decreasing modulus.
Writing $S=\{i_1>\dots>i_{s_1}\}$, we obtain
$$
f_{1}\in\mathscr P^{*}_{T}\!\bigl(u_{1},M_{1}';\varepsilon\bigr),
\quad
f_{2}\in\mathscr P^{*}_{T}\!\bigl(u_{2},M_{2}';\varepsilon\bigr),
$$
where $u_{j}$ collects the $v$-coordinates indexed by $S$ and its
complement, and $M_{j}'$ the corresponding sign patterns.

By Lemma~\ref{cr1}, we get
\begin{equation}\label{eq:NST}
\#\Bigl\{p\in\mathscr P^{*}_{T}(v,m;\varepsilon)\text{ encoded by }S\Bigr\}
\;\ll\;
e^{\bigl(\rho^{*}(u_{1})+\rho^{*}(u_{2})\bigr)T}.
\end{equation}

We claim that
$$
\Delta(S):=\rho^{*}(v)-\bigl(\rho^{*}(u_{1})+\rho^{*}(u_{2})\bigr) >0.
$$

Write $w=(n,n-1,\dots ,1)$ so that $\rho^{*}(v)=w\!\cdot\! v$.
Inside the factors $f_j$ the largest coefficient weight drops from~$n$ to
at most $n-1$, while no weight increases.  Because
$v_{1}>\dots>v_{n}$, we get
$w\!\cdot\! v > w'\!\cdot\! v$, where $w'$ is the modified
weight-vector attached to $(u_{1},u_{2})$, implying the claim.

It now follows that
$$
\#\bigl\{p\in\mathscr P^{*}_{T}(v,m;\varepsilon)
          \text{ reducible}\bigr\}
\;\ll\;
e^{\left( \rho^{*}(v)-\eta \right)T}
$$
where
$
\eta:=\min_{S}\Delta(S)>0
$.
Combined with Lemma \ref{cr1}, this completes the proof.
\end{proof}

\subsection*{Jordan projections of $\Sp_{2n}(\z)$} Let \be\label{sspp} G=\Sp_{2n}(\br)=\{g\in \SL_{2n}(\br): g^t J_n g=J_n\}\quad  J_n=\begin{psmallmatrix}0&\bar I_n\\-\bar I_n&0\end{psmallmatrix}\ee 
where $\bar I_n$ is the anti-diagonal identity matrix.

Then $\fa^+$ is a positive Weyl chamber.
For a loxodromic element $g\in G$, its Jordan projection is given by
$$\lambda(g)=(\lambda_1(g),\cdots, \lambda_n(g), -\lambda_n(g), \cdots, -\lambda_1(g)) \in \inte\fa^+$$ and its eigenvalue datum is
$$\cal E(g)= (m_1(g)e^{\lambda_1(g)},\cdots, m_n(g)e^{\lambda_n(g)}, m_n(g)e^{-\lambda_n(g)}, \cdots, m_1(g)e^{-\lambda_1(g)})$$
where $m_i(g)\in \{\pm 1\}$, $1\le i\le n$.

\begin{thm} (\cite{Ya2}, \cite{MS}, \cite{Ac})\label{every}
    Every integral monic reciprocal polynomial is the characteristic polynomial of some element of $\Sp_{2n}(\z)$.  
\end{thm}

We define  
$\mathsf E_{\Sp_{2n}(\mathbb Z)}(v,m)$ and
$\mathsf E^\star_{\Sp_{2n}(\mathbb Z)}(v,m)$ exactly as in Definition \ref{dir}, replacing $\SL_n(\z)$ by $\Sp_{2n}(\z)$ throughout.

 Observe that $$\rho^*(v)=\rho_{\Sp_{2n}}(v) =\sum_{i=1}^n (n+1-i)v_i $$
 where $\rho_{\Sp_{2n}} $ is the half-sum of all positive roots of $(\frak{sp}_{2n}(\br),\fa)$.
\begin{thm}\label{jm2}
Let $v\in\operatorname{int}\fa^{+}$ and $m=(m_{1},\dots,m_{n})\in\{\pm1\}^{n}$.
For $\e>0$, set
$$
  J^{\Sp}_{T}(v,m;\varepsilon)
  :=\Bigl\{
       (\lambda(\gamma),m(\gamma)):
       \gamma\in\Sp_{2n}(\mathbb Z),\;
       \|\lambda(\gamma)-Tv\|_{\max} \le\varepsilon,\;
       m(\gamma)=m
     \Bigr\}.
$$

For all sufficiently small $\e>0$, we have
$$   \left(\frac{2\e }{ (2n-1) 3^{2n}}\right)^n \le \liminf_{T\to \infty}  \frac{\#   {J}^{\Sp}_{T}(v,m; \e)}{e^{ \rho_{\Sp_{2n}}(v)T} } \le \limsup_{T\to \infty} \frac{\#   {J}^{\Sp}_{T}(v,m; \e)}{e^{ \rho_{\Sp_{2n}} (v)T}}  \le(2\e )^{n}  (n+1)!. $$

Consequently
$$
  \mathsf E_{\Sp_{2n}(\mathbb Z)}(v,m)
  \;=\;
  \rho_{\Sp_{2n}}(v).
$$
\end{thm}

\begin{proof}

For $(\lambda(\gamma),m(\gamma))\in J^{\Sp}_{T}(v,m;\varepsilon)$, set
$$
  p(x)
  \;=\;
  \prod_{i=1}^{n}\bigl(x-m_{i}e^{\lambda_{i}(\gamma)}\bigr)
                   \bigl(x-m_{i}e^{-\lambda_{i}(\gamma)}\bigr)
  \;\in\;
  \mathscr Q^{*}_{T}(v,m;\varepsilon).
$$
The assignment $(\lambda(\gamma),m(\gamma))\mapsto p(x)$ is injective, so
$$
  \#J^{\Sp}_{T}(v,m;\varepsilon)
  \;\le\;
  \#\mathscr Q^{*}_{T}(v,m;\varepsilon).
$$

Conversely, if $p\in\mathscr Q^{*}_{T}(v,m;\varepsilon)$, then
Theorem~\ref{every} and Corollary~\ref{r22} produce a
$\gamma\in\Sp_{2n}(\mathbb Z)$ with
$(\lambda(\gamma),m(\gamma))\in J^{\Sp}_{T}(v,m;\varepsilon)$.
The map $p\mapsto(\lambda(\gamma),m(\gamma))$ is injective, hence
$$
  \#\mathscr Q^{*}_{T}(v,m;\varepsilon)
  \;\le\;
  \#J^{\Sp}_{T}(v,m;\varepsilon).
$$

Hence the claim follows from Theorem \ref{qstar}.
\end{proof}

The lower bound below follows directly from the above theorem and the 
upper bound will be proved in Theorem \ref{up}.
\begin{thm} \label{gen3}
Let $v\in\operatorname{int}\fa^{+}$ and $m=(m_{1},\dots,m_{n})\in\{\pm1\}^{n}$.
For every $0<\varepsilon<1$,
there exist $C_1, C_2>0$ such that
$$C_1 e^{\rho_{\Sp_{2n}} (v)T}\le    \#\Bigl\{ [\ga] \in [\Sp_{2n}(\mathbb Z)] :
              \|\lambda(\ga)-Tv\|\le\e, m(\ga)=m\Bigr\} \le C_2
              e^{2 \rho_{\Sp_{2n}}(v)T} .$$
            In particular,  
\be\label{ee333}
\rho_{\Sp_{2n}}(v) \le 
    \underline{\mathsf E}^{\!\star}_{\Sp_{2n}(\mathbb Z)}(v,m)
    \;\le\;
    \overline{\mathsf E}^{\!\star}_{\Sp_{2n}(\mathbb Z)}(v,m)
   \le 2 \rho_{\Sp_{2n}}(v) .
\ee 
\end{thm}

In \cite{Ya}, Yang establishes a bijection between the set of
$\Sp_{2n}(\mathbb Z)$–conjugacy classes and a distinguished subset of
units of degree $2n$.
Through this correspondence, Theorem~\ref{gen3} can be viewed as a
result about the growth of that collection of algebraic units.

\section{Upper bound for $\mathsf E_\Ga^\star$ for a general lattice}
Let $G$ be a connected semisimple real algebraic group.
Fix a Cartan involution so that
$\frak g=\frak k\oplus\frak p$ is the decomposition into $\pm1$
eigenspaces.
Let $K<G$ be the maximal compact subgroup with Lie algebra $\frak k$,
and let $\frak a\subset\frak p$ be a maximal abelian subalgebra with
closed positive chamber $\frak a^{+}$. Let $\rho_G$ denote the half-sum of all positive roots of $(\frak g,\frak a)$.

Write $A=\exp\frak a$, $A^{+}=\exp\frak a^{+}$, and let
$M=Z_K(A)$.  Every $g\in G$ decomposes as a commuting product
$g=g_h g_e g_u$ of hyperbolic, elliptic and unipotent elements,
and the hyperbolic part $g_h$ is $G$-conjugate to a unique element
$\exp\lambda(g)\in A^{+}$; we call $\lambda(g)$ the
\emph{Jordan projection}.
If $\lambda(g)\in\operatorname{int}\frak a^{+}$ we say $g$ is
\emph{loxodromic}; then $g_u$ is the identity and $g_e$ is conjugate to
an element $m(g)\in M$, unique up to $M$-conjugacy. We denote by $[\Gamma]_{\op{lox}}$ the set of $\Ga$-conjugacy classes of all loxodromic elements of $\Gamma$.

\begin{definition}[Directional entropy for $\Ga$]\label{dir2}
Let $\Gamma<G$ be a lattice. Let $\|\cdot\|$ be any norm on $\fa$.
For any vector $v\in \inte \frak a^{+}$, define the directional \emph{Jordan-entropy}
functions  by
$$
  \overline{\mathsf E}_{\Gamma}(v)
  :=\|v\|\cdot  \lim_{\varepsilon\to0}\;
      \limsup_{T\to\infty}\frac{\log N_\varepsilon(T,v)}{T},\;\;\;
  \underline{\mathsf E}_{\Gamma}(v)
   :=\|v\|\cdot  \lim_{\varepsilon\to0}\;
      \liminf_{T\to\infty}\frac{\log N_\varepsilon(T,v)}{T}
$$
where  $N_\varepsilon(T,v)\:=
     \#\bigl\{\lambda(\gamma):\gamma \in\Gamma:
              \|\lambda(\gamma)-\br_+v\|\le\varepsilon, \|\lambda(\ga)\|\le T \}$.
 
Similarly,
$$  \overline{\mathsf E}^{\!\star}_{\Gamma}(v) 
    :=\|v\|\cdot  \lim_{\varepsilon\to0}\;
      \limsup_{T\to\infty}\frac{\log M_\varepsilon(T,v)}{T},\;\;\;
  \underline{\mathsf E}^{\!\star}_{\Gamma}(v)
    :=\|v\|\cdot  \lim_{\varepsilon\to0}\;
      \liminf_{T\to\infty}\frac{\log M_\varepsilon(T,v)}{T} $$
where $M_\varepsilon(T,v):=
     \#\bigl\{[\gamma]\in[\Gamma]:
              \|\lambda(\gamma)-\br_+ v\|\le\varepsilon,  \|\lambda(\ga)\|\le T \bigr\}.$
These definitions are independent of the choice of a norm.
When the lower and upper values coincide, we write
$\mathsf E_{\Gamma}(v)$ and
$\mathsf E^{\!\star}_{\Gamma}(v)$, respectively.
\end{definition}

\begin{thm}\label{up}
For all $v\in\operatorname{int}\fa^{+}$ and $\e>0$, there exists $C>0$ such that for all $T\ge 1$,
$$\#\Bigl\{\,
     [\gamma]\in[\Gamma] :
     \|\lambda(\gamma)-\br_+ v\|\le\epsilon, \|\lambda(\ga)\|\le T
  \Bigr\} \le C e^{2\rho_G(v) T} .$$

In particular, $$
\overline{\mathsf E}_{\Gamma}^{\star}(v)\le 2\rho_G(v).
$$

\end{thm}

\subsection*{Cartan counting and Upper bound}

Let $\mu(g)\in\fa^{+}$ denote the Cartan projection of $g\in G$, i.e.\
the unique element with $$g\in K e^{\mu(g)} K.$$
If we use the norm on $\fa$ induced from the Killing form on $\frak g$, then for all $g\in G$, we have
$\|\mu(g)\|= d(go, o)$ where $o=[K]\in G/K$ and $d$ is the Riemannian distance on the symmetric space $G/K$. Counting lattice points subject to constraints  on the Cartan projection $\mu(g)$ is considerably better understood than the analogous problem for the Jordan projection; see, for example, (\cite{DRS}, \cite{EM}, \cite{GO}, \cite{BO}, \cite{GN}, etc).
In particular,
following the method of Eskin-McMullen\cite{EM}, we can count lattice points whose Cartan projections lie in prescribed tubes or cones by
combining the mixing of the $A$–action on $\Gamma\backslash G$ with the strong
wavefront lemma stated below.

\begin{lem}[Strong wavefront lemma] \cite[Theorem 3.7]{GO} \label{go} Let $\cal C\subset \inte\fa^+$ be  closed and at positive distance from every wall of $\fa^+$. For any neighborhoods $\cal O_K\subset K$ and $\cal O_A\subset A$ of $e$,
    there exists a neighborhood $U\subset G$ of $e$ such that for any
    $g=k_1 a k_2\in K (\exp \cal C) K$, we have
    $$ U g U\subset k_1 \cal O_K \,a \cal O_A\,  k_2 O_K.$$
\end{lem}
\begin{thm}\label{cartan}
Let $\Gamma<G$ be a lattice and
$v\in\operatorname{int}\fa^{+}$.  For any $\e>0$ we have, as
$T\to\infty$,
$$
\#\bigl\{\ga\in\Gamma :
        \|\mu(\ga)-Tv\|\le\e\bigr\}
   \;\sim\;C\,e^{2\rho_G(v) T}
$$
for some constant $C=C(\e)>0$.
\end{thm}

\begin{proof}
Fix $\e>0$ and put
$$
b_{T,\e}=\bigl\{u\in\fa^{+} : \|u-Tv\|<\e\bigr\},
\qquad
Z_{T}=K\,\exp(b_{T,\e})\,K .
$$
 For $g=k_1(\exp v) k_2\in K (\exp \fa^+) K$,
the Haar measure is
$$dg= \prod_{\alpha} \sinh \alpha(v)  dk_1 dv dk_2 ,$$  where the product runs over all positive roots, counted with multiplicity \cite{Kn}. We obtain that
\be\label{vv} \operatorname{Vol}Z_{T} \sim\;C_{\e}\,e^{2\rho_G(v) T} \ee  for some constant $C_\e>0$. 
Since $v\in \inte\fa^+$, the set $b_{T, \e}$ has a positive distance from all walls of $\fa^+$. Lemma \ref{go} and \eqref{vv} imply that the family $\{Z_{T}\}_{T\gg1}$ is
\emph{well-rounded}: for any $\eta>0$, there exists an open neighborhood $U_\eta$ of $e$ in $G$
such that 
$$Z_{T-\eta} \subset \bigcap_{u_1, u_2 \in U_\eta} u_1Z_Tu_2 \subset \bigcup_{u_1, u_2\in U_\eta } u_1Z_Tu_2\subset Z_{T+\eta} $$
and $$ \limsup_{\eta\to 0} \frac{ \op{Vol}(Z_{T+\eta})} {\op{Vol}(Z_{T-\eta})} = 1 .$$  

Define the counting function $F_T=F_{Z_T}$ on $(\Ga\times \Ga) \ba (G\times G) $ by
$$F_T([g_1], [g_2])=\sum_{\ga\in \G }\chi_{Z_T}( g_1^{-1}  \ga g_2)$$
so that $F_T([e], [e])=\# \Ga\cap Z_T$.
If $\phi_\eta$ is the approximation of the identity function on $ G$ supported on the $\eta$-neighborhood of $e$ in $G$ and $\Phi_\eta([g])=\sum_{\ga\in \Ga} \phi_\eta (\ga g)$, 
then the standard unfolding argument gives that
$$\langle F_T, \Phi_\eta\otimes \Phi_\eta\rangle:= \int  F_T(x_1, x_2) \Phi_\eta (x_1)  \Phi_\eta (x_2)  dx_1 dx_2 =\int_{g\in Z_T} \langle \Phi_\eta , g. \Phi_\eta\rangle_{L^2(\Ga\ba G)} dg. $$
Using strong mixing of the $G$-action on
$L^{2}(\Gamma\backslash G)$ \cite{HM}, we get 
$$\langle F_T, \Phi_\eta\otimes \Phi_\eta\rangle  \sim \frac{1}{\op{Vol}(\Ga\ba G)} \operatorname{Vol}Z_{T}.$$

Noting that  
$$\langle F_{T-\eta} , \Phi_\eta\otimes \Phi_\eta\rangle \le F_T([e], [e])\le  \langle F_{T+\eta}, \Phi_\eta\otimes \Phi_\eta\rangle , $$
 the well-roundness property of the family $\{Z_T\}$ implies that
$$ F_T([e], [e])
      \;\sim\;\frac{1}{\op{Vol}(\Ga\ba G)} \operatorname{Vol}Z_{T} .$$
\end{proof}

The following can be deduced from \cite[Theorem 1.2]{TW} for arithmetic lattices (see the proof of \cite[Theorem\;3.1]{Oh}). For rank one groups, this is a standard fact which follows from the thick-thin decomposition of rank one locally symmetric manifolds of finite volume. Hence by Margulis arithmeticity theorem \cite{Ma}, we get:
\begin{thm}\label{tw}
    Let $\Ga<G$ be a lattice. There exists a compact subset $Q\subset G$ such that
    any compact $AM$-orbit in $\Ga\ba G$ is of the form $\Ga\ba \Ga gAM$ for some $g\in Q$.
\end{thm}

\begin{cor} \label{com}    For any lattice $\Ga<G$, there is $C>1$ such that
    for any conjugacy class $[\ga]\in [\Ga]_{\lox}$,
    there exists $\ga'\in [\ga]$ such that
    $$  \|\lambda(\gamma)-\mu(\gamma')\|\le C .$$
    \end{cor}
\begin{proof} Let $Q$ be a compact subset in Theorem \ref{tw}.
We claim that
there exists a representative $\gamma'\in[\gamma]$ such that
$$
  \gamma' \;=\; g\,e^{\lambda(\gamma)}m_{\gamma}\,g^{-1},
  \qquad
  m_{\gamma}\in M,\; g\in Q.
$$

To see this, since
$\ga$ is loxodromic,  its centralizer in $G$ is of the form $hAMh^{-1}$ with
$\Gamma\backslash\Gamma hAM$ compact \cite{PR}. Since $h=\ga_0 g a_0m_0 \in \Ga Q AM$ with $g_0\in Q$ by Theorem \ref{tw} and $\ga = h e^{\lambda(\ga)}m h^{-1}$ for some $m\in M$,
it suffices to set $$\ga'=g e^{\lambda(\ga)}(m_0 m m_0^{-1})  g^{-1}.$$
 Therefore there is $C>1$ depending only on $Q$ such that
$
  \|\lambda(\gamma)-\mu(\gamma')\|\le C
$
 by
\cite[Lemma 4.6]{Be}.
\end{proof}

Since $\ga'\in [\gamma]$, the map $[\gamma]\to \ga'$ is an injective map to $\Ga$.
Hence we get:
\begin{cor}     Let $\Ga<G$ be a lattice. 
   For any bounded subset $B\subset \fa^+$,
   $$ \# \{[\ga]\in[\Ga] : \lambda(\ga)\in B\} <\infty .$$
\end{cor}

\noindent\textbf{Proof of Theorem \ref{up}.}
Let $B_{\epsilon}(0)\subset\frak a$ be the ball of radius~$\epsilon$
about the origin, and fix
$v\in\operatorname{int}\frak a^{+}$.
Suppose that $\gamma\in\Gamma$ satisfies
$
  \lambda(\gamma)\in T v+B_{\epsilon}(0)
$
for all sufficiently large~$T$.
Then since $v\in \inte\fa^+$ and $T$ is large, $\gamma$ is loxodromic.
Hence, by Corollary \ref{com}, there is $\gamma'\in[\gamma]$ such that
$
  \|\lambda(\gamma)-\mu(\gamma')\|\le C.
$

Thus, by the injectivity of the map $[\gamma]\to \ga'$, 
$$
  \#\bigl\{[\gamma]:
      \lambda(\gamma)\in T v+B_{\epsilon}(0)\bigr\}
  \;\le\;
  \#\bigl\{\gamma'\in\Gamma:
      \mu(\gamma')\in T v+ B_C(0) \bigr\}.
$$
Applying Theorem~\ref{cartan} proves the claim.

\begin{Rmk}\label{qu}
In \cite{Qu}, Quint introduced the growth indicator 
$$\psi_\Ga:\fa^+\to \br \cup\{-\infty\}$$ of a Zariski dense discrete subgroup $\Ga<G$. 
Let $\L_\Ga$ be the limit cone of $\Ga$, that is, the asymptotic cone of the Cartan projection $\mu(\Ga)$. For $v\in \inte\L_\Ga$,
it is equal to
$$\psi_\Ga (v)= \|v\| \inf_{\cal C} \limsup_{T\to \infty} \frac{\log \#\{\ga\in \Ga: \|\mu(\ga)\|\le T, \mu(\ga)\in \cal C\}}{T}$$ where the infimum is taken over all open cones $\cal C\subset \fa^+$ containing $v$. If $\Ga<G$ is a lattice, then $\L_\Ga=\fa^+$ and $\psi_\Ga=2\rho_G$.

While $\psi_\Ga(v)<+\infty$ for all $v\in \fa^+$ and for any discrete subgroup $\Ga$,
the directional entropy $$\overline E^\star_\Ga (v)= \|v\|\cdot  \lim_{\e\to 0}
 \limsup_{T\to \infty} \frac{\log \#\{[\ga]\in \Ga: \|\la (\ga)\|\le T, 
 \| \la (\ga)-\br_+v\|\le \e \}}{T}$$
 may take the value $+\infty$; this already occurs for a normal subgroup of a cocompact lattice of $\SL_2(\br)$ of infinite index.
 Theorem \ref{up} shows that for $\Ga$ lattice,
$ \overline E^\star_\Ga (v)\le \psi_\Ga (v)=2\rho_G(v)$ for all $v\in \inte\fa^+$. It is shown in \cite{CO} that if $\Ga$ is a Zariski dense {\it Borel Anosov} subgroup of $G$, then
$\overline E^\star_\Ga (v)= \psi_\Ga (v)$ for all $v\in \inte\L_\Ga$. 
\end{Rmk} 

\subsection*{ Upper bound without directional restriction}
We will use the following for the upper bound:
\begin{thm} \label{cc3} Let $\Ga<G$ be a lattice in $G$. 
    If $\cal C$ is a convex cone in $\fa^+$ with non-empty interior and $\cal C_T=\{v\in \cal C:\|v\| <T\}$, 
then  $$\# \Gamma\cap K \exp( \cal C_T) K \sim   C\cdot  {e^{ 2 \rho_G(u_{\cal C}) T}} T^{(\op{rank }G-1)/2}$$
where $\|\cdot\|$ is the norm on $\fa$ induced from the Killing form on $\mathfrak g$ and $u_{\cal C} $ is the unique unit vector such that
$2\rho_G(u_{\cal C})=\max_{\|u\|=1, u\in \cal C} 2\rho_G(u)$.
\end{thm}
\begin{proof} 
In \cite[Lemma 5.4]{GO}, it is shown that for $\cal C=\fa^+$,
$$\text{Vol} (K \exp( \cal C_T) K )\sim  C\cdot  {e^{ 2 \rho_G(u_{\cal C})T }} T^{(\op{rank }G-1)/2} .$$  The same proof works for any convex cone $\cal C$ with non-empty interior.

By Theorem \ref{go}, the family
$Z_T= K \exp( \cal C_T) K$, $T\ge  1$ is  well-rounded,  as in the proof of Theorem \ref{cartan}. Consequently, by the same argument used there, we get
$$\# \Gamma\cap K \exp( \cal C_T) K \sim \text{Vol}(Z_T).$$
\end{proof}

\begin{cor} \label{c1} Let $\Ga<G$ be a lattice.
There exist $C>0$ such that for all $T>1$,
$$   
   \#\{[\ga]\in [\Gamma]_{\op{lox}} :  \|\lambda(\ga)\| <T\} \le C  
    e^{2\|\rho_G\| T}  T^{(\op{rank }G-1)/2}  $$
    where $\|\rho_G\|=\max_{u\in \fa^+, \|u\|=1} \rho_G(u)$.
\end{cor}
\begin{proof}
Let $[\ga]\to \ga'$ be the injective map from the conjugacy classes of loxodromic elements
to $\Gamma$ given in Corollary \ref{com}. 
Therefore $$ \#\{[\ga]\in [ {\Gamma}]_{\op{lox}}:  \|\lambda(\ga)\| <T\}
\le \# \{ \ga' \in \Gamma, \|\mu (\ga')\| <T +C\} $$
where $C>1$ is as in Corollary \ref{com}. Therefore the upper bound follows from Theorem \ref{cc3}. \end{proof}
We remark that in \cite{DL}, some upper bound  for cocompact lattices of $G$ was obtained. We record the following for $\SL_n(\z)$: 
\begin{cor} \label{euc}
There exist $C_1, C_2>0$ such that  for all $T>1$,
   \begin{multline*}
  C_1 e^{ d_n T/2}  \le 
   \#\{[\ga]\in [ \SL_n(\z)]_{\op{lox}}:  \|\lambda(\ga)\|_{\op{Euc}} <T\} \le C_2  
    T^{(n-2)/2} e^{ d_n T}   \end{multline*} 
    where $d_n={\sqrt{\frac {n(n^2-1)}{3}}}$.
\end{cor}

\begin{proof} The lower bound follows from Corollary \ref{euc3}.
Since the norm on $\fa$ induced by the Killing form on $\frak{sl}_n(\br)$
is a constant multiple of the Euclidean norm on $\fa$,
the upper bound follows from Corollary \ref{c1}
and \eqref{vs}. 
\end{proof}

  \end{document}